\newif\ifdraft

\let\ifdraft\iffalse %

\ifdraft
    \documentclass[reqno]{amsart}						%
    \usepackage{microtype}
    \usepackage{fullpage}
\else
    \documentclass[smallcondensed]{svjour3}
    \smartqed  

    \makeatletter
    \def\cl@chapter{\@elt {theorem}}
    \makeatother
\fi

\usepackage{placeins}

\ifdraft
    \usepackage{amsmath}						%
    \usepackage{amssymb}						%
    \usepackage{amsfonts}						%
    \usepackage{amsthm}
    \usepackage[foot]{amsaddr}						%
    \usepackage{array}

    \usepackage{mathtools}						%
    \mathtoolsset{%
    }
\else
    \usepackage{amsmath}
    \usepackage{amsthm}
    \usepackage{amssymb}
\fi

\ifdraft
    \usepackage[proportional,lining]{libertine}
    \usepackage[libertine,libaltvw,cmintegrals,varbb]{newtxmath}
\fi

\usepackage{dsfont}							%

\usepackage[%
cal=cm,
]
{mathalfa}

\usepackage[labelfont={bf,small},labelsep=colon,font=small]{caption}	%

\usepackage{algorithm}
\usepackage[noend]{algpseudocode}

\usepackage[dvipsnames,svgnames, table]{xcolor}						%
\colorlet{MyBlue}{DodgerBlue!75!Black}
\colorlet{MyGreen}{DarkGreen!85!Black}

\usepackage{subcaption}
\usepackage{tikz}							%
\usetikzlibrary{calc,patterns}
\usepackage{pgf}
\usepackage{pgfplotstable}
\usepgfplotslibrary{external}
\usepackage{acronym}						%
\usepackage{booktabs}						%
\usepackage{latexsym}						%
\usepackage{paralist}						%
\usepackage{xspace}							%
\usepackage{xifthen}                        %
\usepackage{etoolbox}

\ifdraft
    \usepackage[sort&compress]{natbib}					%

    \def\bibhang{24pt}

\fi

\usepackage{hyperref}
\hypersetup{
final,
colorlinks=true,
linktocpage=true,
pdfstartview=FitH,
breaklinks=true,
pdfpagemode=UseNone,
pageanchor=true,
pdfpagemode=UseOutlines,
plainpages=false,
bookmarksnumbered,
bookmarksopen=false,
bookmarksopenlevel=1,
hypertexnames=true,
pdfhighlight=/O,
urlcolor=Maroon,linkcolor=MyBlue!60!black,citecolor=DarkGreen!70!black,	%
pdftitle={},
pdfauthor={},
pdfsubject={},
pdfkeywords={},
pdfcreator={pdfLaTeX},
pdfproducer={LaTeX with hyperref}
}

\numberwithin{equation}{section}						%
\usepackage[sort&compress,capitalize,nameinlink]{cleveref}			%

\ifdraft
    \theoremstyle{plain}
    \newtheorem{theorem}{Theorem}						%
    \newtheorem*{theorem*}{Theorem}						%
    \newtheorem{corollary}[theorem]{Corollary}					%
    \newtheorem*{corollary*}{Corollary}						%
    \newtheorem{lemma}[theorem]{Lemma}					%
    \newtheorem{proposition}[theorem]{Proposition}				%

    \theoremstyle{definition}
    \newtheorem{definition}[theorem]{Definition}				%
    \newtheorem*{definition*}{Definition}					%
    \newtheorem{assumption}{Assumption}					%
    \newtheorem*{assumption*}{Assumptions}					%

    \theoremstyle{remark}
    \newtheorem*{remark*}{Remark}						%
    \newtheorem{example}{Example}						%
    \newtheorem*{example*}{Example}						%
\else
    \newtheorem{assumption}{Assumption}
\fi

\ifdraft
\fi

\numberwithin{equation}{section}						%
\numberwithin{theorem}{section}						%
\numberwithin{lemma}{section}						%
\numberwithin{proposition}{section}						%
\numberwithin{definition}{section}						%
\numberwithin{corollary}{section}						%
\numberwithin{remark}{section}						%
\numberwithin{example}{section}						%

\def\addlegendimage{\csname pgfplots@addlegendimage\endcsname}

\newcommand{\revise}[1]{#1}		%
\newcommand{\revisebis}[1]{#1}		%

\DeclareMathOperator*{\ri}{ri}

\newcommand{\RR}{\mathbb{R}}
\newcommand{\bbR}{\mathbb{R}}         %

\newcommand{\B}{\mathcal{B}}
\newcommand{\N}{\mathcal{N}}
\providecommand{\U}{U}
\renewcommand{\U}{\mathcal{U}} %

\DeclareMathOperator*{\dist}{dist}

\newcommand{\rank}{\operatorname{rank}}

\DeclareMathOperator{\Tr}{Tr}

\DeclareMathOperator{\sign}{sign}

\ifdef{\C}{                          %
    \renewcommand{\C}{\mathcal{C}}
}{
    \newcommand{\C}{\mathcal{C}}
}

\newcommand{\funman}{F} %
\newcommand{\funns}{g}  %

\newcommand{\funtot}{F}     %
\newcommand{\funspart}{f}   %
\newcommand{\funnspart}{g}  %

\DeclareMathOperator*{\argmin}{argmin}
\newcommand{\prox}{\mathbf{prox}}
\newcommand{\PGop}{\mathsf{T}} %

\newcommand{\manupdate}{\mathrm{ManAcc}}

\ifdef{\M}{                          %
    \renewcommand{\M}{\mathcal{M}}
}{
    \newcommand{\M}{\mathcal{M}}
}

\ifdef{\B}{                          %
    \renewcommand{\B}{\mathcal{B}}
}{
    \newcommand{\B}{\mathcal{B}}
}

\newcommand{\proj}{\mathrm{proj}}
\newcommand{\distM}[1][]{%
  \ifthenelse{\isempty{#1}}{{\dist}_{\M}}{{\dist}_{#1}}
}

\newcommand{\smoothcurve}{c}
\newcommand{\geocurve}{\gamma}

\newcommand{\grad}{\operatorname{grad}}
\newcommand{\Hess}{\operatorname{Hess}}

\newcommand{\R}{\operatorname{R}}           %
\newcommand{\D}{\operatorname{D}}           %

\newcommand{\tangentBundle}{T \mathcal B}

\newcommand{\tangent}[2]{T_{#1} #2}
\newcommand{\tangentM}[1][]{%
  \ifthenelse{\isempty{#1}}{T_x \M}{T_{#1} \M}%
}

\newcommand{\tangentMstar}[1][]{%
  \ifthenelse{\isempty{#1}}{T_x \M^\star}{T_{#1} \M^\star}%
}

\newcommand{\normalM}[1][]{%
  \ifthenelse{\isempty{#1}}{N_x \M}{N_{#1} \M}%
}

\usepackage{colortbl}
\usepackage{makecell}%
\pgfplotstableset{
    every even row/.style={before row={\rowcolor[gray]{0.9}}},
    col sep=&,
    empty cells with={--}, %
    columns/algo/.style={string type, column name={Algorithm}},
    columns/tol/.style={column name={Tolerance}},
    columns/F-F*/.style={column name={$F(x_k)-F(x^\star)$}},
    columns/proxgrad/.style={column name={\makecell{\#prox. grad.\\ steps}}},
    columns/gradF/.style={column name={\makecell{\#$\manupdate$\\ steps}}},
    columns/HessF/.style={column name={\#$\Hess F(\cdot)[\cdot]$}},
    columns/f/.style={column name={\#$f$}},
    columns/g/.style={column name={\#$g$}},
    every head row/.style={before row=\toprule,after row=\midrule},
    every last row/.style={after row=\bottomrule},
}

\usepackage[normalem]{ulem}
\usepackage{wrapfig}

\begin{document}

\title{Newton acceleration on manifolds identified by proximal gradient methods\footnote{This work is partly funded %
by the ANR JCJC project \emph{STROLL} (ANR-19-CE23-0008).}}

\ifdraft

    \author[G.~Bareilles]{Gilles Bareilles$^{\ast}$}
    \email{gilles.bareilles@univ-grenoble-alpes.fr}

    \author[F.~Iutzeler]{Franck Iutzeler$^{\ast}$}
    \email{franck.iutzeler@univ-grenoble-alpes.fr}

    \author[J.~Malick]{\\J\'er\^ome Malick$^{\diamond}$}
    \email{jerome.malick@univ-grenoble-alpes.fr}

    \address{$^{\ast}$  Univ. Grenoble Alpes, LJK, 38000 Grenoble, France}
    \address{$^{\diamond}$ Univ. Grenoble Alpes, CNRS, LJK, 38000 Grenoble, France}

\else
    \author{Gilles Bareilles, Franck Iutzeler, J\'er\^ome Malick}
    \titlerunning{Newton acceleration of the proximal gradient methods}
    \institute{G. Bareilles \and F. Iutzeler \and %
                J. Malick \at Univ. Grenoble Alpes, CNRS, Grenoble INP, LJK, 38000 Grenoble, France\\
                \email{firstname.lastname@univ-grenoble-alpes.fr}   \\
    }

    \date{Received: date / Accepted: date}
\fi

\maketitle
\begin{abstract}
    Proximal methods are known to identify the underlying substructure of nonsmooth optimization problems. Even more, in many interesting situations, the output of a proximity operator comes with its %
    structure {at} no additional cost, and convergence is improved once \revise{it matches the structure of a minimizer.}
    However, it is impossible in general to know whether the current structure is final or not; such highly valuable information has to be exploited adaptively. To do so, we place ourselves in the case where a proximal gradient method can identify manifolds of differentiability of the nonsmooth objective. Leveraging this manifold identification, we show that Riemannian Newton-like methods can be intertwined with the proximal gradient steps to drastically boost the convergence. We prove the superlinear convergence of the algorithm when solving some nondegenerated nonsmooth nonconvex optimization problems. We provide numerical illustrations on optimization problems regularized by $\ell_1$-norm or trace-norm.
\end{abstract}
\section{Introduction}

Nonsmoothness naturally appears in various applications of optimization, e.g. in decomposition methods in operations research \cite{briant-lemarechal-etal-2008} %
or in sparsity-inducing regularization techniques in data analysis\;\cite{bach2012optimization}.
In these applications, 
the \emph{nonsmooth} objective functions usually present a \emph{smooth} substructure,
which involves smooth submanifolds on which
the functions are locally smooth. To fix ideas, consider the simple example of the $\ell_1$ norm: %
though nonsmooth, it is obviously smooth around any point when restricted to the vector space of points with the same support.

Exploiting the underlying smooth substructure of objective functions to develop second-order methods has been a subject of fruitful %
research in nonsmooth optimization, pioneered by the developments around $\mathcal{U}$-Newton algorithms \cite{lemarechal2000u} and the notion of partial smoothness \cite{lewis2002active}. Let us mention the $\mathcal{UV}$-Newton bundle method of\;\cite{mifflin2005algorithm},
and the recent $k$-bundle Newton method of\;\cite{lewis2019simple}.
Interestingly, these Newton-type methods for nonsmooth optimization are connected to the standard Newton methods of nonlinear programming (SQP) and to the Newton methods of Riemannian optimization; see \cite{miller2005newton}. 

In this paper, we focus on a special situation where the smooth substructure
can be exploited numerically.
We consider the nonsmooth optimization problem
\begin{equation}\tag{$\mathcal P$} \label{pb:compositepb}
    \min_{x\in\mathbb R^n} F(x) \triangleq f(x)+g(x)
\end{equation}
where $f$ is a smooth differentiable function, and $g$ is not everywhere differentiable -- but admits a \emph{simple} proximal operator. More precisely, we assume that the proximal operator of $g$ outputs an explicit expression of the proximal point together with a representation of the current active submanifold. Coming back to the example of the $\ell_1$-norm: its proximity operator %
puts exactly to $0$ some coordinates of the input vector after a comparison test; hence, the output has some sparsity structure, which is known as a byproduct of the computation. More generally, this situation covers a large class of 
applications, where $g$ is used to enforce some prior structure such as sparsity of vectors (when $g$ is one of the $\ell_1,\ell_{0.5},\ell_0$-norms) or low rank of matrices (when $g$ is the nuclear norm); see e.g.\;\cite{bach2012optimization}.

Since $g$ has a simple proximal operator, first-order methods %
to minimize $F$ are the (accelerated) proximal gradient algorithms\revisebis{; see \cite[Chap.~10]{beck2017first} for a general reviews of these methods and their analysis}. %
Interestingly, in nondegenerate cases, the iterates produced by these algorithms eventually reach the optimal submanifold (ie. the manifold which contains the minimizer): it is the so-called \emph{identification} property of proximal algorithms, extensively studied in %
general settings; we refer to \cite{burke1988identification}, \cite{wright1993identifiable}, \cite{drusvyatskiy2014optimality}, or \cite{lewis2018partial}. For $\ell_1$-norm regularization, this means that after a finite but unknown number of iterations the algorithm ``identifies'' the final set of non-zero variables; see the pedagogical paper\;\cite{iutzeler2020nonsmoothness} for further discussions.

In the ideal case where we know that the iterates are %
on the optimal manifold, one could switch to a more sophisticated method, e.g.\,updating parameters of first-order methods as in\;\cite{liang2017activity}, considering Riemannian Newton methods as in\;\cite{daniilidis2006geometrical}, \revise{or other second-order schemes as in\;\cite{lewisProximalMethodComposite2016,lee2020accelerating}}. Unfortunately, even though we know the current structure of the iterates %
and we know that they will identify the optimal manifold in finite time, we \emph{never} know if the current manifold is the optimal one. %

We propose here a Newton acceleration\footnote{\revisebis{We choose the term ``Newton acceleration'' to emphasize the similarity with the celebrated Nesterov acceleration \cite{nesterov1983method}. Indeed both methods add an acceleration step after the proximal gradient iteration. But, unlike Nesterov's method where the acceleration is provided by an inertial step, the Newton acceleration comes from a second-order step on a smooth substructure, as we detail in this paper.}} of the proximal gradient algorithm solving the nonsmooth optimization problem\;\eqref{pb:compositepb}, that adaptively uses identification. 
\revisebis{ 
Our algorithm uses the same basic ingredients that work behind the scenes for existing nonsmooth second-order algorithms (e.g. \cite{mifflin2005algorithm} and \cite{lewis2019simple}):
(i) nonsmooth structure identification 
and (ii) efficient Newton-type methods to benefit from faster convergence along this structure. 
However, we rely on the explicit proximity operator of $g$ to benefit from exact structure identification, contrary to the approximate identification of the above methods. In addition, we perform Riemannian Newton steps on the identified manifold, while previous methods do not leverage its tractable Riemannian nature. 
We present a convergence analysis showing superlinear convergence of the resulting algorithm under some qualification assumptions -- but without prior knowledge on the final optimal submanifold. 
}
Finally, we provide numerical illustrations showing the interests of the proposed Newton acceleration on typical structure-inducing regularized problems (sparse logistic regression and low-rank least-squares).
Along the way, our study reveals results that have some interest on their own, in particular: we \revise{refine} the smoothness properties of the proximal gradient operator around structured critical points; we formalize complementary properties on \revise{line searches} in Riemannian optimization; \revise{we also bring a careful attention to the technical details induced by nonconvexity.}

\revisebis{Let us finally note that the Newton acceleration of the proximal gradient that we propose here should not be confused with proximal-Newton schemes such as \cite{lee2014proximal,becker2019quasinewton,aravkin2022proximal}. These methods essentially replace the gradient step by a (quasi-)Newton step before applying a proximity operator. Hence, they do not explicitly use the second order information of the function $g$ brought by its nonsmooth structure, which is instrumental in our developments.}

The paper is organized as follows.
First, in \cref{sec:prelim} we recall the useful notions of Riemannian optimization and variational analysis. %
Then, we introduce in \cref{sec:algo} our template algorithm alternating a proximal gradient step with a Riemannian update on the identified manifold. 
In \cref{sec:specificationsManUp}, we specify the implementation of efficient Riemannian Newton-type methods and illustrate their performances in \cref{sec:num}. The paper also contains three appendices with material
 used in our proofs;
some of these %
results are well-known and just recalled here, but several others seem to be less-known or not precisely treated in the %
literature. %

\section{Preliminaries: definitions, recalls, and examples}
\label{sec:prelim}

In this section, we introduce the notions
which will be central in our developments.
Our notation and terminology follow closely those of the monographs \cite{absil2009optimization} for Riemannian optimization
and \cite{rockafellar2009variational} for nonsmooth optimization. This section can be skipped by readers familiar with these topics.

\vspace*{-2ex}
\subsection{Recalls on Riemannian optimization}\label{sec:RiemannianRecalls}

We briefly introduce below the tools of Riemannian optimization used in this paper. We refer the reader to \cite{absil2009optimization} and \cite{boumal2022intromanifolds} for more extensive presentations.
\revise{In the rest of the paper, $\M$ denotes a submanifold of $\bbR^{n}$ or $\bbR^{m \times n}$.}

\medskip
\noindent\textbf{Submanifolds.}
A subset $\M$ of $\RR^n$ is said to be a \emph{$p$-dimensional $\C^2$-submanifold} of $\RR^n$ around $\bar{x} \in \M$ if there exists %
$\C^2$ function $\varphi : \RR^p \to \RR^n$ such that $\varphi$ \revise{maps} a neighborhood of $0 \in \RR^p$ \revise{to} a neighborhood of $\bar{x} \in\M$, \revise{that admits a smooth (local) inverse,} and \revise{which derivative} at $\varphi^{-1}(\bar{x})=0$ is injective.
A $p$-dimensional $\C^2$-submanifold of $\RR^n$ can alternatively be defined via a local equation, that is, a $\C^2$ function $\Phi : \RR^n \to \RR^{n-p}$ with a surjective derivative at $\bar x \in \M$ that satisfies for all $x$ close enough to $\bar x$: $x \in\M \Leftrightarrow \Phi(x) = 0$.

A basic tool to investigate approximations on manifolds is notion of the \emph{smooth curves}. A smooth curve on $\M$ is a $\C^2$ application $\gamma : I\subset \RR \to \M\subset \RR^n$, where $I$ is an open interval containing $0$. At each point $x\in\M$, the \emph{tangent space}, noted $\tangentM$, can be defined as the velocities of all smooth curves passing by $x$ at $0$:
\begin{equation*}
    \tangentM \triangleq \left\{ \smoothcurve'(0) ~|~ \smoothcurve : I \to \M \text{ is a smooth curve around $0$ and } \smoothcurve(0)=x \right\}.
\end{equation*}
The tangent space is a $p$-dimensional %
space containing \emph{tangent vectors}. Each tangent space $\tangentM$ is equipped with a scalar product $\langle\cdot, \cdot\rangle_x : \tangentM\times\tangentM\to\RR$, and the associated norm $\|\cdot\|_x$. In many cases, the tangent metric varies smoothly with $x$, making the manifold \emph{Riemannian}. In this paper, we use the ambient space scalar product to define the scalar product on tangent spaces; we will thus drop the subscript in the tangent scalar product and norm notations when there is no confusion possible. Related to the tangent space, we will also consider the \emph{normal space} $\normalM$ at $x\in\M$, defined as the orthogonal space to $\tangentM $ in $\RR^n$, and the \emph{tangent bundle manifold} defined by:
\[
\tangentBundle \triangleq \bigcup_{x\in\M}(x,\tangentM).
\]
Note also that both tangent and normal spaces at $x\in\M$  admit explicit expressions from derivatives of local parametrization $\varphi$ or local equations $\Phi$ defining $\M$:
\begin{equation*}
    \tangentM = \text{Im}~ \D_\varphi (0) = \text{Ker}~ \D_\Phi (x) \qquad \normalM = \text{Ker}~ \D_\varphi (0)^* = \text{Im}~ \D_\Phi (x)^*
\end{equation*}

A \emph{metric} on $\M$ can be defined as the minimal length over all curves joining two points $x, y\in\M$, ie. $\distM(x, y) = \inf_{\smoothcurve \in C_{x,y}}  \int_0^1 \|\smoothcurve'(t)\|_{\smoothcurve(t)} \mathrm{d}t$, where $C_{x,y}$ is the set of $[0, 1]\to\M$ smooth curves $\smoothcurve$ such that $\smoothcurve(0)=x$, $\smoothcurve(1)=y$. The minimizing curves generalize the notion of straight line between two points to manifolds. The constant speed parametrization of any minimizing curve is called a \emph{geodesic}.

\medskip
\noindent\textbf{Riemannian Gradients and Hessian.}
Let $\funman:\M\to\RR$, the \emph{Riemannian differential} of $\funman$ at\;$x$ is the linear operator $\D \funman(x):\tangentM\to\RR$ defined by $\D \funman(x)[\eta]\triangleq\left. \frac{\mathrm{d}}{\mathrm{d}t}\funman\circ \smoothcurve(t) \right|_{t=0}$, where $\smoothcurve$ is a smooth curve such that $\smoothcurve(0)=x$ and $\smoothcurve'(0)=\eta$. In turn, the \emph{Riemannian gradient} $\grad \funman(x)$ is the unique vector of $\tangentM$ such that, for any tangent vector\;$\eta$, $\D \funman(x)[\eta] = \langle \grad \funman(x), \eta \rangle$.
If $\grad \funman(x)$ exists, a first order Taylor development can be formulated. Let $x\in\M$, $\eta\in\tangentM$ and $\smoothcurve$ denote a smooth curve passing by $x$, with velocity $\eta$ at $0$; then, for $t$ near $0$,
\begin{equation*}\label{eq:manifoldFirstOrderDevSmoothCurve}
    \funman\circ \smoothcurve(t) = \funman(x) + t \langle \grad \funman(x), \eta \rangle + o(t).
\end{equation*}

Notions of derivation for vector fields and of acceleration for curves are used to define second-order objects. Let a curve $\smoothcurve:I\to\M$ and a smooth vector field $Z$ on $\smoothcurve$, ie.\;a smooth map such that $Z(t)\in\tangentM[c(t)]$ for $t\in I$. The \emph{covariant derivative} of $Z$ on the curve $\smoothcurve$, denoted $\frac{\D}{\mathrm{d}t}Z:I\to\tangentBundle$, is defined by $\frac{\D}{\mathrm{d}t}Z(t) \triangleq \proj_{\smoothcurve(t)} Z'(t)$, where $Z'(t)$ denotes the derivative in the ambient space $\bbR^n$ and $\proj_{x}$ corresponds to the orthogonal projector from $\RR^n$ to $\tangentM$. The \emph{acceleration} of a curve $\smoothcurve$ is defined as the covariant derivative of its velocity: $\smoothcurve''(t) \triangleq \frac{\D}{\mathrm{d}t}\smoothcurve'(0)$. 

The \emph{Riemannian Hessian}  of $\funman$ at $x$ along $\eta$ is the linear operator $\Hess \funman(x):\tangentM\to\tangentM$ defined by the relation $\Hess \funman(x)[\eta] \triangleq \left. \frac{\D}{\mathrm{d}t} \grad \funman(\smoothcurve(t)) \right|_{t=0}$, where $\smoothcurve$ is a smooth curve such that $\smoothcurve(0)=x$, $\smoothcurve'(0)=\eta$. 
Equivalently, we have
$\langle\Hess \funman(x)[\eta], \eta\rangle = \left. \frac{\mathrm{d}^2}{\mathrm{d}t^2}\funman\circ \geocurve(t) \right|_{t=0}$, where $\geocurve$ is a geodesic such that $\geocurve(0)=x$, $\gamma'(0) = \eta$.
A second order Taylor development can now be formulated. Let $x\in\M$, $\eta\in\tangentM$, and $\smoothcurve$ be a smooth curve such that $\smoothcurve(0)=x$, $\smoothcurve'(0)=\eta$. Then, for $t$ near $0$,
\begin{equation*}%
    \funman\!\circ\!\smoothcurve(t) = \funman(x) \!+ t \langle \grad \funman(x), \eta \rangle + \frac{t^2}{2}\!(\langle \Hess \funman(x)[\eta], \eta \rangle\!+\!\langle \grad \funman(x), c''(0)\rangle)+ o(t^2).
\end{equation*}
If $\funman:\M\to\RR$ has a smooth extension on $\RR^n$, the Riemannian gradient and Hessian can be computed from their Euclidean counterparts:
for a smooth function\;$\bar \funman:\RR^n\to\RR$ that coincides with $\funman$ on $\M$, %
\begin{equation}\label{eq:egrad_to_rgrad}
    \grad \funman(x) = \proj_{x}(\nabla \bar{\funman}(x)),
\end{equation}
and, for $\bar G:\bbR^n\to\bbR^n$ a smooth mapping that coincides with $\grad \funman$ on $\M$, %
\begin{equation}\label{eq:ehess_to_rhess}
    \Hess \funman(x)[\eta] = \proj_x \left( \D\bar G(x)[\eta] \right).
\end{equation}

\noindent\textbf{Algorithms on manifolds: retractions and convergence rates.}
Iterative Riemannian methods require a way to produce curves on $\M$ given a point $x$ and a tangent vector $\eta$. A geodesic curve passing at $(x, \eta)$, while attractive as the generalization of the straight line, has a prohibitive computational cost. We thus \emph{retractions}, i.e. approximations of it, defined on a manifold $\M$ as a smooth map $\R : \tangentBundle \to \M$ such that
\[%
    \text{$\R_x (0) = x\quad$ and $\quad \D\R_x (0) : \tangentM \to\tangentM$ is the identity map: $D\R_x (0 )[v] = v$,}
\]%
where, for each $x \in\M$, $\R_x : \tangentM\to\M$ is defined as the restriction of $\R$ at $x$, so that $\R_x(v) = \R(x, v)$. A \emph{second-order retraction} is a retraction $\R$ such that, for all $(x, \eta)\in\tangentBundle$, the curve $c(t)=\R_x(t\eta)$ has zero acceleration at 0: $c''(0)=0$.
Thus $t\mapsto\R_x(t\eta)$ is a practical curve passing by $(x, \eta)$ at $0$, and provides a similar development as above: %
for $t$ near 0,
\begin{equation}\label{eq:manifoldSecondOrderDevSecRetraction}
    \funman\circ\R_x(t \eta) = \funman(x) + t \langle \grad \funman(x), \eta\rangle + \frac{t^2}{2}\langle \Hess \funman(x)[\eta], \eta\rangle + o(t^2 \|\eta\|^2).
\end{equation}

Finally, the convergence rates on manifolds are defined as follows. A sequence of points $(x_k)$ \emph{converges (Q-)linearly} to some point $\bar{x}\in\M$ if there exist an integer $K>0$ and a constant $q\in(0,1)$ such that, for all $k\ge K$, there holds
\begin{equation*}
    \distM(x_{k+1}, \bar{x})\le q~ \distM(x_{k}, \bar{x}).
\end{equation*}
The sequence \emph{converges with order at least $p$} if there exists an integer $K>0$ and a constant $q\in(0,1)$ such that, for all $k\ge K$, there holds
\begin{equation*}
    \distM(x_{k+1}, \bar{x})\le q~\distM(x_{k}, \bar{x})^p.
\end{equation*}
The convergence is \emph{superlinear} when $p>1$ and \emph{quadratic} when $p=2$.

\medskip
\noindent\textbf{Examples of submanifolds and related objects.}
In this paper, we will illustrate our developments with two sparsity-inducing norms (see Section\;\ref{sec:ex}) involving respectively the two following manifolds.

\begin{example}[Fixed coordinate-sparsity subspaces]\label{ex:manifold_l1}
    We consider the submanifold %
    \begin{equation}
        \label{eq:Mi}
        \M_I \triangleq \{ x\in\bbR^n : x_i = 0 \text{ for } i\in I \},
    \end{equation}
    where $I \subset \{1, \dots, n\}$. This manifold is actually a vector space and all related notions have simple expressions, as follows.

     The tangent space at any point identifies with the manifold itself: $\tangentM_I = \M_I$. The orthogonal projection of a vector $d\in\bbR^n$ on the tangent space writes $\proj_x(d)$, where $[\proj_x(d)]_i$ is $d_i$ if $i\not\in I$, and null otherwise. The map $\R_x(\eta) = x+\eta$ defines a second-order retraction.
    Given a function $\funman$ defined on the ambient space, the Riemannian gradient and Hessian-vector product of the restriction of $F$ to $\M_I$ are obtained from their Euclidean counterparts by a simple projection: for $x, \eta\in\tangentBundle$,
    \begin{equation*}
        \grad \funman(x) = \proj_x(\nabla \funman(x)) \qquad \Hess \funman(x)[\eta] = \proj_x(\nabla^2 \funman(x)[\eta]).
    \end{equation*}
\end{example}

\begin{example}[Fixed rank matrices]\label{ex:manifold_fixedrank}
    We consider the manifold of fixed-rank matrices
    \begin{equation}
        \label{eq:Mr}
        \M_r \triangleq \{ x\in\bbR^{m\times n} : \rank(x) = r\},
    \end{equation}
    for which we refer to \cite[Sec.\;7.5]{boumal2022intromanifolds}. A rank-$r$ matrix $x\in\M_r$ is represented as $x = U \Sigma V^\top$, where $U\in\RR^{m \times r}$, $V\in\RR^{n \times r}$, $\Sigma\in\bbR^{r \times r}$ such that $U^\top U = I_n$, $V^\top V = I_m$ and $\Sigma$ is a diagonal matrix with positive entries. Such a decomposition can be obtained by computing the %
    singular value decomposition of the matrix $x$. Using this representation, a tangent vector $\eta\in\tangentM_r$ writes
    \begin{equation*}
        \eta = UMV^\top + U_p V^\top + U V_p^\top,
    \end{equation*}
    where $M\in\bbR^{r\times r}$, $U_p\in\bbR^{m\times r}$, $V_p\in\bbR^{n\times r}$ such that $U^\top U_p = 0$, $V^\top V_p = 0$. The orthogonal projection of a vector $d\in\RR^{m\times n}$ onto $\tangentM_r$ writes $\proj_x(d) = d - U^\top d V$.
    Given a function $\funman$ defined on the ambient space, a Riemannian gradient and Hessian-vector product of $F$ restricted to $\M_r$ can be obtained from their Euclidean counterparts: for $x, \eta\in\tangentBundle$, and with 
    $P_U^\top = I_m - UU^\top$, $P_V^\top = I_n - VV^\top$.
    {\footnotesize
        \begin{align*}
            \grad \funman(x) &= \proj_x(\nabla F(x)) \\
            \Hess \funman(x)[\eta] &= \proj_x(\nabla^2 \funman(x)[\eta]) + \left[ P_U^\top \nabla \funman(x) V_p \Sigma^{-1} \right]V^\top + U\left[ P_V^\top \nabla \funman(x)^\top U_p \Sigma^{-1} \right]^\top.
        \end{align*}
    }
\end{example}

\vspace*{-3ex}
\subsection{Recalls on nonsmooth optimization}

We review the basic notions of variational analysis used in this paper, following the monograph\;\cite{rockafellar2009variational}.
 For this section, $\funns\colon\bbR^n\to\bar{\bbR}=\RR\cup\{+\infty\}$ is a proper function.

\medskip
\noindent\textbf{Subgradients.}
Consider a point $\bar{x}$ with $\funns(\bar{x})$ finite. The set of \emph{regular subgradients} %
\begin{equation*}
    \widehat\partial \funns(\bar{x}) \triangleq \left\{ v : \funns(x) \ge \funns(\bar{x}) + \langle v, x-\bar{x} \rangle + o(\|x-\bar{x}\|) \text{ for all } x\in\bbR^n \right\}
\end{equation*}
is closed and convex, but the subdifferential mapping $\widehat\partial \funns(\cdot)$ may not be outer semi-continuous \cite[Th. 8.6, Prop. 8.7]{rockafellar2009variational}. To overcome this problem, the set of \emph{(general or limiting) subgradients} is defined as
\begin{equation*}
    \partial \funns(\bar{x}) \triangleq \left\{ \lim_r v_r : v_r\in\widehat\partial \funns(x_r), x_r\to\bar{x}, \funns(x_r)\to \funns(\bar{x}) \right\}.
\end{equation*}
The limiting subdifferential is by design outer semi-continuous:
\[
\limsup_{x\to\bar{x}} \partial \funns (x) = \{u : \exists x_r\to\bar{x}, \exists u_r \to u \text{ with } u_r\in  \partial \funns (x_r) \} ~\subset~ \partial \funns (\bar{x}),
\]
which is an attractive property to study the properties of sequences of points whose subgradients converge.
We say that a function is \emph{(Clarke) regular} at $\bar{x}$ if the regular and limiting subdifferentials at $\bar{x}$ coincide \cite[Def. 7.25, Cor. 8.11]{rockafellar2009variational}. This is notably the case for convex functions where the two above definitions coincide with the convex subdifferential \cite[Prop. 8.12]{rockafellar2009variational}.

\medskip
\noindent\textbf{Optimality conditions and critical points.}
The subdifferential allows to derive optimality conditions: for a local minimizer $\bar x$ of $F$, we have $0\in \partial F(\bar x)$. For the objective function of \eqref{pb:compositepb}, this writes
\[
0\in \nabla f (\bar x) + \partial g(\bar x) \qquad\text{or equivalently}\qquad -\nabla f (\bar x) \in \partial g(\bar x).
\]
A point satisfying these conditions is called a critical point. The analysis of the algorithms of this paper will provide convergence guarantees towards critical points. %

\medskip
\noindent\textbf{Proximity operator.}
A central tool to tackle non-differentiable functions is the \emph{proximity operator}.
For\;$\gamma\!>\!0$ and a function $\funns$; it is defined as the set-valued mapping
\begin{equation*}
    \prox_{\gamma \funns}(y) \triangleq \argmin_{u\in\bbR^n} \left\{ \funns(u) + \frac{1}{2\gamma}\|u-y\|^2\right\}.
\end{equation*}
Since this operator will be at the core of our future developments, \emph{we will assume that it is non-empty for all $y$}. Note that this is a reasonable assumption since it is satisfied as soon as $\funns$  is lower-bounded\footnote{The weaker assumption %
of prox-boundedness (ie. $\funns+r\|\cdot\|^2$ is bounded below for\;some\;$r$) implies that $ \prox_{\gamma \funns}(y) $ is non-empty %
when $\gamma$ is taken sufficiently small; see \cite[Chap. 1.G]{rockafellar2009variational}.}, %
which is trivially verified by our functions of interest (see Section~\ref{sec:ex}). Though computing proximal points is in general difficult, it is easy for some relevant cases %
as the $\ell_1$-norm or the trace-norm; see Section\;\ref{sec:ex}.%

\medskip
\noindent\textbf{Prox-regularity.}
A function $\funns$ is \emph{prox-regular} at a point $\bar{x}$ for a subgradient $\bar{v} \in \partial \funns(\bar{x})$ if $\funns$ is finite, locally lower semi-continuous at $\bar{x}$, and there exists $r>0$ and $\varepsilon>0$ such~that
$%
    \funns(x') \ge \funns(x) + \langle v, x'-x\rangle - \frac{r}{2}\|x'-x\|^2
$ %
whenever $v \in \partial \funns(x)$, $\|x-\bar{x}\|<\varepsilon$, $\|x'-\bar{x}\|<\varepsilon$, $\|v-\bar{v}\|<\varepsilon$ and $\funns(x)<\funns(\bar{x})+\varepsilon$.
When this holds for all $\bar{v} \in\partial \funns(\bar{x})$, we say that $\funns$ is prox-regular at $\bar{x}$ \cite[Def. 13.27]{rockafellar2009variational}.

This property allows to have local Lipschitzness of the proximal operator as well as its characterization by first-order optimality conditions; see \cite[Th. 4]{hare2009computing} and \cref{th:proxex}. Specifically, we will use that if $\funns$ is $r$-prox-regular at $\bar{x}$, then, for any $\gamma<1/r$, $\prox_{\gamma \funns}(y)$ is single-valued and Lipschitz continuous for any $y$ near $\bar{x}+\gamma \bar{v}$ where $ \bar{v}\in\partial \funns(\bar{x})$ and $\bar{x} = \prox_{\funns/r}(\bar{x}+ \bar{v}/r)$. Furthermore, in this neighborhood, it is uniquely determined by the relation
 $
    x = \prox_{\gamma \funns}(y)  ~\Leftrightarrow~  \frac{y-x}{\gamma} \in \partial \funns(x),
$
which characterizes proximal maps using first-order optimality conditions. %

\vspace*{-1ex}
\subsection{Running examples}
\label{sec:ex}

\begin{example}[$\ell_1$ norm]
    In the context of \Cref{ex:manifold_l1}, we consider the $\ell_1$ norm defined on $\bbR^n$ as $ \|x\|_1 = \sum_{i=1}^n |x_i|$. This function is convex, thus prox-regular at every point with $r=0$. %
    Its proximity operator admits a closed form: %
    \begin{equation*}
        [\prox_{\gamma \|\cdot\|_1}(y)]_i =
        \begin{cases}
            y_i + \gamma &\text{ if } y_i<-\gamma\\
            0 &\text{ if } -\gamma \le y_i \le \gamma\\
            y_i - \gamma &\text{ if } y_i>\gamma
        \end{cases}
    \end{equation*}
    which naturally gives sparse outputs. In other words,  $x= \prox_{\gamma \|\cdot\|_1}(y)$ lies on  $\M_{{I}}$ (see \eqref{eq:Mi}) where ${I}$ is the complementary of support of $x$.
Observe also that the restriction of $\|\cdot\|_1$ to the manifold $\M_{{I}}$ is locally smooth.
The $\ell_1$ norm thus admits a Riemannian gradient and Hessian at point ${x}$:
    \begin{equation*}
        \grad \|\cdot\|_1({x}) = \sign({x}) \hspace*{0.7cm} \text{ and }  \hspace*{0.7cm}  \Hess \|\cdot\|_1 ({x}) = 0,
    \end{equation*}
    where $\sign(x)\in\{-1, 0, 1\}$ denotes the sign of $x$, null when $x=0$.
\end{example}

\begin{example}[nuclear norm]
    Following the notation of \cref{ex:manifold_fixedrank}, we consider the nuclear norm, defined on $\bbR^{m\times n}$ as
    $
        \|x\|_* = \sum_{i=1}^{\rank(x)} \Sigma_{ii},
    $
    where $\Sigma$ denotes the diagonal term of the singular value decomposition of $x$. This function is convex, and thus prox-regular at every point with $r=0$. Its proximity operator admits a closed form: for matrix $y$ ($=U\Sigma V^\top$),
    \begin{equation*}
        \prox_{\gamma \|\cdot\|_*}(y) = U (\Sigma - \gamma)_+ V^\top,
    \end{equation*}
    where the coefficient $(i,j)$ of $(\Sigma - \gamma)_+$ is defined as $\max(\Sigma_{i j}-\gamma, 0)$.
    Thus, $x = \prox_{\gamma \|\cdot\|_*}(y)$ has low rank, by construction. Said otherwise, $x$ lies on $\M_{r}$ (see \eqref{eq:Mr}) where $r = \rank(\Sigma - \gamma)_+$. Observe also that the restriction of the nuclear norm to the manifold $\M_{r}$  is locally smooth, and thus admits a Riemannian gradient and Hessian at point $x$: denoting $\eta=U M V^\top + U_p V^\top + U V_p^\top\in\tangentM_{r}$ a tangent vector,
    \begin{align*}
        \grad \|\cdot\|_*(x) &= U V^\top \\
        \Hess \|\cdot\|_*(x)[\eta] &= U \left[\tilde{F} \circ (M-M^\top) \right] V^\top + U_p \Sigma^{-1}V^\top + U\Sigma^{-1}V_p^T,
    \end{align*}
    where $\circ$ denotes the Hadamard product and $\tilde{F}\in\RR^{{\bar{r}}\times {\bar{r}}}$ is such that $\tilde{F}_{ij} = 1 / (\Sigma_{jj} + \Sigma_{ii})$ if $\Sigma_{jj} \neq \Sigma_{ii}$, and $\tilde{F}_{ij}=0$ otherwise. This statement is proved in \cref{sec:diffsvd}.
\end{example}

\section{General proximal algorithm with Riemannian acceleration }\label{sec:genealgo}
\label{sec:algo}

As mentioned in the introduction and in the previous examples,
the output of a proximity operator often comes with the knowledge of the current manifold on which it lives. In this section, we leverage this ability to an algorithmic advantage by reducing our working space to the identified structure. ``Smooth'' structures (involving smooth submanifolds and smooth restrictions on it) are of special interest and open the way to Newton acceleration.

Let us start by specifying 
the blanket assumptions on the %
problem\;\eqref{pb:compositepb}. These assumptions are mostly common except the third point which directly comes from our idea of using the proximal operator both for the optimization itself and as an oracle for the current structure of the iterates.
\begin{assumption}\label{asm:fun}
    The functions $f$ and $g$ are proper and 
    \begin{itemize}
        \item[i)] $f$ is $\C^2(\bbR^n)$ with an $L$-Lipschitz continuous gradient;
        \item[ii)] $g$ is lower semi-continuous;
        \item[iii)] $\prox_{\gamma g}$ is non-empty on $\mathbb R^n$ for any $\gamma>0$;
        \item[iv)] $F(x) = f(x) + g(x)$ is bounded below.
    \end{itemize}
\end{assumption}

In this setup, we propose a general algorithm (Algorithm\;\ref{alg:generic}) which consists in, first, performing a proximal gradient step $x_{k} \in \prox_{\gamma g}(y_{k-1} - \gamma \nabla f(y_{k-1}))$ that provides both the current point $x_k$ and the manifold  $\M_k$ where it lies, and, second, carrying out a Riemannian optimization update $\manupdate_{\M_k}$ on the current manifold. This algorithm is {general} in the sense that we do not precise for now what is the Riemannian step $\manupdate$. 

We start in Section\;\ref{sec:proxgrad} with a technical result about the local smoothness of the proximal gradient operator. In Section\;\ref{sec:identif}, we analyze the %
identification property of this algorithm. In Section\;\ref{sec:super}, we study how Riemannian methods with local superlinear convergence %
propagate their rate to %
\cref{alg:generic}. 
We will investigate later in Section\;\ref{sec:specificationsManUp} the Riemannian Newton acceleration falling into this scheme.

\begin{algorithm}
    \caption{General structure exploiting algorithm\label{alg:generic}}
    \begin{algorithmic}[1]
        \Require \revise{Pick $x_{0}$ arbitrary, $\gamma < 1 / L$.}
        \Repeat
        \State Compute $x_{k}\in\prox_{\gamma g}(y_{k-1} - \gamma \nabla f(y_{k-1}))$ and get $\M_{k}\ni x_{k}$ %
        \State Update $y_{k} = \manupdate_{\M_{k}} (x_{k})$ on the current manifold %
        \Until{stopping criterion}
    \end{algorithmic}
\end{algorithm}

\vspace*{-2ex}
\subsection{Smoothness and localization of the proximal gradient}\label{sec:proxgrad}

The results of this section are built on $g$ being a partly smooth function; see \cite{lewis2002active}.
\begin{definition}[partial smoothness]
    A function $g$ is ($\C^2$-)\emph{partly smooth} at a point $\bar{x}$ relative to a set $\M$ containing $\bar{x}$ if $\M$ is a $\C^2$ manifold around $\bar{x}$ and:
    \begin{itemize}
        \item (smoothness) the restriction of $g$ to $\M$ is a $\C^2$ function near $\bar{x}$;
        \item (regularity) $g$ is (Clarke) regular at all points $x\in\M$ near $\bar{x}$, with $\partial g(x)\neq\emptyset$;
        \item (sharpness) the affine span of $\partial g(\bar{x})$ is a translate of $\normalM[\bar{x}]$;
        \item (sub-continuity) the set-valued mapping $\partial g$ restricted to $\M$ is continuous at $\bar{x}$.
    \end{itemize}
\end{definition}

Under this assumption, we show in the next theorem that the proximal gradient smoothly locates active manifolds: if some input $\bar{y}$ is mapped onto $\M$, then the proximal gradient is $\M$-valued and $\mathcal{C}^1$ around $\bar{y}$. This result is based on the sensitivity analysis of partly smooth functions \cite[Sec. 5]{lewis2002active}. \revise{The proof extends and refines the rationale of \cite[Th.\;28]{daniilidis2006geometrical} and \cite[Th.\;4.4]{poliquin1996prox} that deal with the proximity operator. We use this extension to allow for a full stepsize range of $(0,1/r)$ in the proximal gradient around any point $\bar{x}$.}

\begin{theorem}[Proximal gradient points smoothly locate manifolds] \label{th:PGpointlocatemanifolds}
    Let $f$ be a $\C^2$ function on $\RR^n$ and $g$ a lower semi-continuous function on $\RR^n$. Suppose that $g$ is both $r$-prox-regular at $\bar{x}$ and partly-smooth relative to $\M$ at $\bar{x}$.

    Take $\gamma$, $\bar \gamma$ such that $0<\gamma<\bar \gamma \leq 1/r$ and $\bar{x}=\prox_{\bar\gamma g}(\bar{y}-\bar\gamma \nabla f(\bar{y}))$. If
    \begin{itemize}
        \item[i)] $\frac{1}{\gamma}(\bar{y}-\bar{x}) - \nabla f(\bar{y})  \in \ri \partial g(\bar{x})$ ~\revise{(the relative interior of the subdifferential at $\bar{x}$)};
        \item[ii)] either a) $\gamma$ is sufficiently close to $\bar{\gamma}$, or b) $\bar{y}$ is sufficiently close to $\bar{x}$;
    \end{itemize}
    then, the proximal gradient $y \mapsto \prox_{\gamma g}(y - \gamma\nabla f(y))$ is $\C^{1}$ and $\M$-valued near $\bar{y}$.
\end{theorem}
\begin{proof}
    Adopting the same reasoning as in \cite[Sec. 5]{lewis2002active} and \cite[Sec. 4.1]{daniilidis2006geometrical}, we consider the function
    \begin{equation*}
        \begin{array}{rl}
            \rho : \bbR^n\times\bbR^n & \to \bbR                                                 \\
            (x, y)                    & \mapsto g(x)+\frac{1}{2\gamma}\|x-y+\gamma\nabla f(y)\|^2,
        \end{array}
    \end{equation*}
    and denote by $\rho_y = \rho(\cdot,y)$. Computing the proximal gradient $\prox_{\gamma g}(y-\gamma\nabla f(y))$ can then be seen as minimizing the parametrized function $\rho_y$.

    \noindent\underline{Step 1.} As a first step, we study the minimizers of $\rho_y$ restricted to $\M$, for $y$ near\;$\bar{y}$. %
    We consider the parametric manifold optimization problem, for $y$ near $\bar{y}$:
    \begin{equation}\tag{$ P_\M(y)$}
        \min_{x\in\M} \rho_y(x). \label{pb:pb_mancstr}
    \end{equation}
    Since $g$ is $\C^2$-partly-smooth relative to $\M$ and $f$ is $\C^2(\bbR^n)$, $\rho_y$ is twice continuously differentiable on $\M$. \revise{Moreover, the $r$-prox-regularity gives easily (see \cref{lemma:prox_strong_min}) that $\rho_{\bar{y}}$ is lower-bounded by $(\frac{1}{\gamma}-r)\|\cdot - \bar{x}\|^2/2$ on a neighborhood of $\bar x$ in $\bbR^n$ and, a fortiori, in $\M$. From usual rationale (see e.g.
    \cite[Chap 4.2, 6.1]{boumal2022intromanifolds}), this implies}
    \begin{equation*}
        \grad \rho_{\bar{y}}(\bar{x}) = 0 \qquad \Hess \rho_{\bar{y}}(\bar{x}) \succeq \Big(\frac{1}{\gamma}-r\Big) I \succ 0,
    \end{equation*}
    \revise{which are the conditions to apply the implicit functions theorem, as follows.}

    We consider the equation $\Phi(x, y) = 0$, for $x, y$ near $\bar{x}, \bar{y}$, where $\Phi:\M\times\bbR^n\to\tangentBundle$ is defined as $\Phi(x, y) = \grad \rho_y(x)$. This function is continuously differentiable on a neighborhood of $(\bar{x}, \bar{y})$, and its differential relative to $\bar x$ at that point, $\Hess \rho_{\bar{y}}(\bar{x})$, is invertible. The implicit function theorem thus grants the existence of neighborhoods $\N_{\bar{x}}$, $\N_{\bar{y}}$ of $\bar{x}$, $\bar{y}$ in $\M$, $\bbR^n$, and a continuously differentiable function $\hat{x}:\N_{\bar{y}} \to \N_{\bar{x}}$ such that, for any $y$ in $\N_{\bar{y}}$, $\Phi(\hat{x}(y), y) = \grad \rho_y(\hat{x}(y)) = 0$. Actually, $\hat{x}(y)$ is a strong minimizer of $\rho_y$ on $\M$ for $y$ close enough to $\bar{y}$. Indeed, the mapping $\hat{x}$ is continuous on $\N_{\bar{y}}$, so that $y \mapsto \Hess \rho_y(\hat{x}(y))$ is also continuous there and the property $\Hess \rho_{\bar{y}}(\hat{x}(\bar{y}))\succ 0$ extends locally around $\bar{y}$.

    \noindent\underline{Step 2.}  As a second step, we turn to show that the minimizer $\hat{x}(y)$ of $\rho_y$ on $\M$ is actually a strong critical point of $\rho_y$ in $\bbR^n$ \cite[Def. 5.3]{lewis2002active}, and thus the proximal gradient of point $y$. More precisely, we claim that, for $y$ near $\bar{y}$ and $x=\hat{x}(y)$, there holds $0\in\ri \partial \rho_y(x)$, that is
    \begin{equation*}
        \frac{1}{\gamma}(y-x) - \nabla f(y)  \in \ri \partial g(x).
    \end{equation*}

    This property holds at $(\bar{x}, \bar{y})$ by assumption. By contradiction, assume there exist sequences of points $(y_r)$ with limit $\bar{y}$, $(x_r) = (\hat x(y_r))$ with limit $\bar{x} = \hat x(\bar{y})$ and $(h_r)$ of unit norm $\|h_r\|=1$ such that for all $r$, $h_r$ separates 0 from $\partial \rho_{y_r}(x_r)$:
    \begin{equation*}
        \inf_{h\in\partial \rho_{y_r}(x_r)} \langle h_r, h\rangle \ge 0.
    \end{equation*}
    Since $(h_r)$ is bounded, a converging subsequence can be extracted from it, let $\bar{h}$ denote its limit. At the cost of renaming iterates, we assume that $\lim_{r\to\infty} h_r = \bar{h}$. The above property still holds at the limit $r\to\infty$. Indeed, let $\bar{u}\in\partial \rho_{\bar{y}}(\bar{x})$. Since $g$ is partly smooth, the mapping $(x, y)\in \N_{\bar{x}}\times\N_{\bar{y}} \mapsto\partial \rho_y(x) = \partial g(x) + \frac{1}{\gamma}(x-y)$ is continuous. Therefore, there exists a sequence $(u_r)$ such that $u_r\in\partial\rho_{y_r}(x_r)$ and $\lim_{r\to\infty} u_r=\bar{u}$. We have for all $r$: $\langle u_r, h_r\rangle \ge 0$, which yields at the limit $\langle\bar{u}, \bar{h}\rangle\ge 0$. Thus $\bar{h}$ separates $0$ from $\partial \rho_{\bar{y}}(\bar{x})$, which contradicts our assumption.

    \noindent\underline{Conclusion.} We thus have a continuously differentiable function $\hat{x}$ defined on a neighborhood of $\bar{y}$ such that i) $\hat{x}(\bar{y}) = \bar{x}$, ii) $\hat{x}(y)$ is a strong minimizer of $\rho_y$ on $\M$, iii) $0\in\ri\partial \rho_y(\hat{x}(y))$.

    This last point tells us that $(y-\hat x(y))/\gamma -\nabla f(y) \in \partial g(\hat x(y))$. The characterization of proximity by the optimality condition (\cref{th:proxex}) gives
that $\hat{x}(y)=\prox_{\gamma g}(y-\gamma \nabla f(y))$ for $y$ close enough to $\bar{y}$.
\end{proof}

\begin{figure}[!ht]
    \centering
    \includegraphics[width=0.4\linewidth]{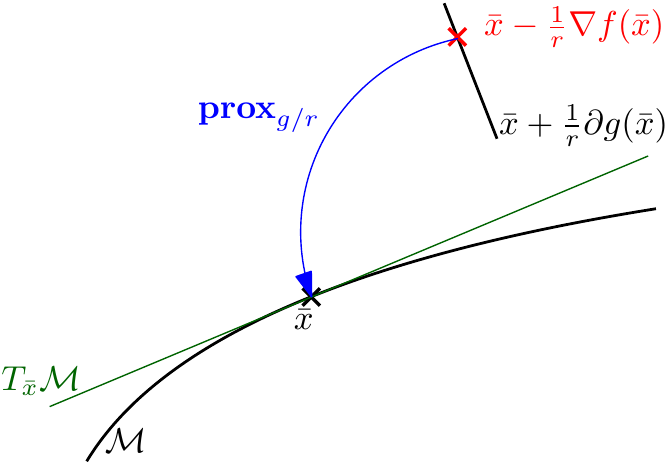}
    \caption{Illustration of a $r$-structured critical point. Point i) is illustrated by the blue arrow, and point ii) implies that the red cross is in the interior of the black segment. Partial smoothness appears in the fact that the black segment is perpendicular to the tangent plane of $\M$ at\;$\bar{x}$.\label{fig:struct}}
    \vspace*{-2ex}
\end{figure}

\vspace*{-2ex}
\subsection{Structure identification}\label{sec:identif}

\cref{th:PGpointlocatemanifolds} captures the localization properties of the proximal gradient operator. It also enables us to precisely define a condition under which a point can be localized. We formalize it
 in the definition of \emph{$r$-structured critical points}, an illustration of which is depicted on \cref{fig:struct}.

\begin{definition}\label{def:structqual}
   A point $\bar{x}$ of a $C^2$ submanifold $\M$ is a $r$-\emph{structured critical} point for $(f,g)$~if we have:
    \begin{itemize}
        \item[i)] proximal gradient stability: $\bar{x} = \prox_{g/r}(\bar{x}-1/r \nabla f( \bar{x} ) )$ ;
        \item[ii)] qualification condition: $0 \in \ri (\nabla f + \partial g)(\bar{x})$;
        \item[iii)] prox-regularity: $g$ is $r$-prox-regular at $\bar{x}$;
        \item[iv)] partial smoothness: $g$ is partly-smooth at $\bar{x}$ with respect to $\M$.
    \end{itemize}
\end{definition}

While %
ii),iii),iv) are %
standard in the literature (see e.g.\;\cite{daniilidis2006geometrical}), %
i) is not always explicited (an exception is for the notion of identifiability in~\cite{drusvyatskiy2014optimality}). %
It is directly verified when $g$ is convex (for any $r>0$), but this is not the case  when $g$ is nonconvex.\footnote{\revise{The following example shows that in the nonconvex setting, ii) and iii) do not necessarily imply i). Take $f$ null and $g$ as follows, then the proximity operator of $g$ at $0$ writes:
\begin{equation*}
    g(x) = \begin{cases}
      x^{2}/2 & \text{if } |x| \le 1\\
      1-3x/2 & \text{if } x \ge 1\\
      1+3x/2 & \text{if }x \le -1
    \end{cases},
    \qquad
    \prox_{\gamma g}(0) = \begin{cases}
      0 & \text{if } \gamma \in (0, 8/9) \\
      \{-3\gamma / 2, 0, 3\gamma / 2\} & \text{if } \gamma = 8/9 \\
      \{-3\gamma / 2, 3\gamma / 2\} & \text{if } \gamma > 8/9.
    \end{cases}
\end{equation*}
The function $g$ is $1$-prox-regular at $0$, there holds $0 \in \ri \partial g(0) = \{0\}$, and yet $0$ is not a fixed point of the proximal operator with stepsizes close to $1$.}}

Using this notion and \cref{th:PGpointlocatemanifolds}, we get the precise identification result of the proximal gradient algorithm, that we need in the forthcoming analysis.

\begin{corollary}[Identification]\label{coro:identif}
    Let $f$ be a $\C^2$ function on $\RR^n$ and $g$ a lower semi-continuous function on $\RR^n$. Take $\bar{x}\in\M$ a $r$-structured critical point for $(f,g)$. Then, for any $\gamma \in (0, 1/r)$, if the sequence $(y_k)$ satisfies $y_k\to\bar{x}$, then $x_k\triangleq\prox_{\gamma g}(y_k-\gamma \nabla f(y_k)) \in \M $ for $k$ large enough.
\end{corollary}

\begin{proof}
    The notion of $r$-structured critical point allows us to apply 
    \cref{th:PGpointlocatemanifolds} with $\bar{x}=\bar{y}\in\M$ and $\bar \gamma = 1/r$. So we get that, for any $\gamma \in (0, 1/r)$, the proximal gradient map $y \mapsto \prox_{\gamma g}(y - \gamma\nabla f(y))$ is $\C^{1}$ and $\M$-valued near $\bar{x}$. Since its input $(y_k)$ converges to $\bar{x}$, the proximal gradient mapping reaches the neighborhood in finite time, which guarantees that $(y_k)$ are $\M$-valued.
\end{proof}

\subsection{Superlinear convergence} %
\label{sec:super}

Using the structure identification result above, we can guarantee that our method benefits from superlinear convergence, provided that the considered Riemannian method is superlinearly convergent locally around a limit point. %

\begin{theorem}%
    \label{th:identiflocalcv_sperlin}
    Let \cref{asm:fun} hold and take $\gamma\in(0, 1/L)$\revise{, where $L$ is the Lipschitz constant for $\nabla f$}. Assume that \cref{alg:generic} generates a sequence $(y_k)$ which admits at least one limit point $\bar{x}$ such that:
    \begin{itemize}
        \item[i)] $\bar{x}\in\M$ is a $r$-structured critical point for $(f,g)$ with $r<1/\gamma$;
        \item[ii)] $\manupdate_{\M}$ has superlinear convergence rate of order $1+\theta\in(1,2)$ %
        near $\bar{x}$ in $\M$.
    \end{itemize}
    Then, after some finite time:
    \begin{itemize}
        \item[a)] the full sequence $(x_k)$ lies on $\M$;
        \item[b)] $x_k$ converges to $\bar{x}$ superlinearly with the same order as $\manupdate$:
              \begin{equation}\label{eq:super}
                  \distM(x_{k+1},\bar{x}) \le c ~ \distM(x_k, \bar{x})^{1+\theta} \qquad \text{ for some $c>0$.}
              \end{equation}
    \end{itemize}
\end{theorem}

\begin{proof}
    Let us note $ \PGop(y) = \prox_{\gamma g}(y - \gamma \nabla f(y)) $ for $y\in\mathbb{R}^n$.
    The part i) of the assumptions enables us to show the existence of some neighborhood of $\bar{x}$ on which the proximal gradient operation is $\M$-valued and Lipschitz continuous. More precisely, 
    \cref{th:PGpointlocatemanifolds} 
    implies that there exists $\delta_1>0$ and $C>0$ such that, 
    \[%
      \PGop(y)\in\M \label{proof:inM} \quad \text{ and } \quad \|\PGop(y)- \PGop(\bar{x})\| \le C \|y-\bar{x}\| \qquad \text{for all $y$ in $\B(\bar{x}, \delta_1)$}.
    \]%
    Now, if $y$ belongs to $\M$, we get that there exists $\varepsilon_1>0$ such that for any $y$ in $\B_{\M}(\bar{x}, \varepsilon_1)$, $\PGop(y)\in\M$; but in addition, the Euclidean Lipschitz continuity can be translated into a Riemannian one (see \cref{lemma:Euclidean_to_Riemannian}) since   for some $\delta>0$,
    \begin{multline}\label{eq:EuctoRie}
        (1-\delta)\distM(\PGop(y), \bar{x}) = (1-\delta)\distM(\PGop(y), \PGop(\bar{x}))  \le \|\PGop(y)-\PGop(\bar{x})\| \\
        \le C \|y-\bar{x}\| \le C (1+\delta)\distM(y, \bar{x})
    \end{multline}
    Hence, there is $q_1>0$ such that for any $y$ in $\B_{\M}(\bar{x}, \varepsilon_1)$
    \begin{equation}\label{proof:PGopEstimate}
        \distM(\PGop(y), \bar{x}) =  \distM(\PGop(y), \PGop(\bar{x}))   \le  q_1 ~ \distM(y, \bar{x}).
    \end{equation}

    Then, the part ii) of the assumptions gives us the existence of $\varepsilon_2$, $q_2>0$ and $\theta\in(0,1)$ such that, for any $x$ in $\B_{\M}(\bar{x}, \varepsilon_2)$,
    \begin{equation}\label{proof:manupdateEstimate}
        \distM(\manupdate_{\M}(x), \bar{x}) \le q_2 ~ \distM(x, \bar{x})^{1+\theta}.
    \end{equation}
    Let us now take any $x\in\B_{\M}(\bar{x}, \varepsilon)$ where $\varepsilon=\min(\varepsilon_1,\varepsilon_2,(\varepsilon_1 / q_2)^{\frac{1}{1+\theta}},(q_2 q_1)^{-\frac{1}{\theta}})$: \\
    \noindent (\emph{i}) Since $x\in\B_{\M}(\bar{x}, \varepsilon_2)$, the manifold update \eqref{proof:manupdateEstimate} yields
    \begin{equation*}
        \distM(\manupdate_{\M}(x), \bar{x}) \le q_2 ~ \distM(x, \bar{x})^{1+\theta} \le q_2 ~ \varepsilon^{1+\theta} \le \varepsilon_1.
    \end{equation*}
    \noindent (\emph{ii}) As $\manupdate_{\M}(x)$ lies in $\B_{\M}(\bar{x}, \varepsilon_1)$, the proximal gradient update \eqref{proof:PGopEstimate} applied to $y = \manupdate_{\M}(x)$ gives
    \begin{multline}\label{eq:superlinear}
        \distM(\PGop(\manupdate_{\M}(x)), \bar{x}) \le    q_1 \distM(\manupdate_{\M}(x), \bar{x}) \\ \le q_1 q_2 ~ \distM(x, \bar{x})^{1+\theta} \le q_1 q_2 ~ \varepsilon^\theta \distM(x, \bar{x}).
    \end{multline}
    Since $q_2 q_1 \varepsilon^\theta \le 1$  by construction, this gives %
    \begin{equation}\label{eq:manstable}
        \distM(\PGop(\manupdate_{\M}(x)), \bar{x}) \le \distM(x, \bar{x}) \qquad \text{for any $x\in\B_{\M}(\bar{x}, \varepsilon)$}.
    \end{equation}
    We have thus proved the existence of a neighborhood $\B_\M(\bar{x}, \varepsilon)$ of $\bar{x}$ in $\M$ which is stable for an iteration of \cref{alg:generic} and over which one iteration has a superlinear improvement of order $1+\theta$ (by \eqref{eq:superlinear}).

    Finally, since $\bar{x}$ is a limit point of $(y_k)$, there exists $K<\infty$ such that $y_K\in\B(\bar{x},(1-\delta)\varepsilon/C)$. Besides, \eqref{eq:EuctoRie} tells us that $ \distM(\PGop(y_K), \bar{x})\le \varepsilon$ and thus $x_k$ and $y_k$ belong to $\B_{\M}(\bar{x}, \varepsilon)$ for all $k>K$ by \eqref{eq:manstable}. %
    We conclude that $x_{k+1}=\PGop(y_k)\in\M$ for all $k\geq K$, and, using \eqref{eq:superlinear}, that we have \eqref{eq:super} with $c= q_1 q_2$, for all $k>K$. 
\end{proof}

\section{Newton acceleration}\label{sec:specificationsManUp}

In this section, we investigate the possibilities of manifold acceleration %
within \cref{alg:generic}.
We show in Sections \ref{sec:newt} and \ref{sec:truncnewt} how to use Riemannian (truncated) Newton accelerations within our framework and derive superlinear/quadratic convergence guarantees. %
\revise{A technical difficulty to ensure global convergence when interlacing proximal gradient updates with Riemannian Newton accelerations is to guarantee some functional decrease. Thus, we first study in Section\;\ref{sec:linesearch} the use of \revise{line search} for $\manupdate_\M$  in our context.}

\vspace*{-1ex}
\subsection{Ensuring functional descent while preserving local rates: \revise{line search}}
\label{sec:linesearch}

We use in the following convergence proofs three properties of $\manupdate_\M$: it should produce an update that lives on $\M$, enjoy a superlinear local convergence rate, and not degrade function value. For this last point, we consider a simple \revise{line search} and we prove that, under mild assumptions,
it helps to find a point which decreases function value, and retains the favorable local properties. Surprisingly, this result does not appear in the standard references on Riemannian optimization. We provide here the necessary developments inspired from the classical monograph\;\cite{dennis1996numerical}.

Standing at point $x\in\M$ with a proposed direction $\eta\in\tangentM$, a stepsize $\alpha>0$ is \emph{acceptable} \revise{if} it satisfies the following \emph{Armijo} condition %
\begin{equation}\label{eq:Armijo}
    \funtot(\R_x(\alpha \eta)) \le \funtot(x) + m_1 \alpha \langle \grad \funtot(x), \eta\rangle, \qquad\text{for $0 < m_1 < 1/2$.}
\end{equation}
The line search employs a second-order retraction $\R_x$, \revisebis{e.g.\;the exponential map, a projection retraction \cite{Absil2012ProjectionlikeRO}, or any other second-order retraction \cite{boumal2022intromanifolds}.\footnote{\revisebis{Indeed, in many applications of Riemannian optimization, computing geodesics and the exponential map can be costly and then retractions provide an efficient alternative. For this reason, we consider here second-order retractions \cite{absil2009optimization,boumal2022intromanifolds}.}}} %
The conditions under which stepsizes satisfying the Armijo rule exist are discussed in \cite[Sec 6.3]{dennis1996numerical}, the following lemma can then be derived.

\begin{lemma}\label{lemma:LS_exists_valid_stepsizes}
    Let \cref{asm:fun} hold and consider a manifold $\M$ equipped with a retraction $\R$ and a pair $(x, \eta)\in\tangentBundle$. If $\funtot$ is differentiable on $\M$ at $x$, $\langle \grad \funtot (x), \eta\rangle <0$, and $m_1<1$, then there exists $\hat\alpha>0$ such that any step size $\alpha\in(0, \hat\alpha)$ is acceptable by the Armijo rule \eqref{eq:Armijo}.
\end{lemma}

\begin{proof}
    We adapt a part of the proof of  \cite[Th. 6.3.2]{dennis1996numerical} for the Armijo rule and the Riemannian setting. Since $m_1<1/2$, for any $\alpha$ sufficiently small there holds
    \begin{equation*}
        \funtot\circ \R_x(\alpha \eta) \le \funtot\circ\R_x(0) + m_1 \D\left(\funtot\circ\R_x \right)(0)[\alpha\eta] = \funtot(x) + m_1 \alpha \langle \grad \funtot(x), \eta\rangle.
    \end{equation*}
    Since $\funtot$ is bounded below, there exists a smallest $\hat{\alpha}$ such that $\funtot(\R_x(\hat\alpha \eta)) = \funtot(x) + m_1 \hat\alpha \langle \grad \funtot(x), \eta\rangle$. Thus all stepsizes in $(0, \hat\alpha)$ are acceptable by\;\eqref{eq:Armijo}.
\end{proof}

In addition, a \revise{line search} performed near a minimizer with a Newton direction
should accept the unit stepsize, so that a full step may be taken. This is the case when the Riemannian Hessian around this minimizer is positive definite as stated by the next lemma, \revise{which is a direct corollary of \cref{th:unitstepsuperlinearcv}.}
\begin{lemma}\label{lemma:LS_unit_stepsize}
    Let \cref{asm:fun} hold and consider a manifold $\M$ equipped with a retraction $\R$, a point $x^\star\in\M$ and a pair $(x, \eta)\in\tangentBundle$. Assume that $\funtot$ is twice differentiable on $\M$ near a strong local minimizer  $x^\star$  on $\M$, that is $\Hess \funman(x^\star)$ is positive definite.
    If the direction $\eta$ brings a superlinear improvement towards $x^\star$, that is $\distM(\R_x(\eta), x^\star) = o(\distM(x, x^\star))$ as $x\to x^\star$, and $0<m_1<1/2$, then $\eta$ is acceptable by the Armijo rule\;\eqref{eq:Armijo} with unit stepsize $\alpha=1$.
\end{lemma}

In the following, we will consider a \emph{backtracking} \revise{line search} for finding an acceptable stepsize $\alpha$: the unit stepsize is first tried, and then the search space is reduced geometrically. %
In practice, we use exactly \cite[Alg. A6.3.1]{dennis1996numerical}, which features polynomial interpolation of $\funman$ in the search space.

\vspace*{-1ex}
\subsection{Riemannian Newton \& quadratic convergence}
\label{sec:newt}

We construct a manifold update based on the Riemannian Newton method \cite[Chap.\;6]{absil2009optimization}, which is the simplest method with a local quadratic convergence. It consists in finding $d\in\tangentM$ that minimizes the second order model \eqref{eq:manifoldSecondOrderDevSecRetraction} of $\funtot$ at point $x\in\M$, or equivalently that solves Newton equation; see \cite[Sec. 6.2]{boumal2022intromanifolds}.

\begin{algorithm}
    \caption{$\manupdate$-Newton\label{alg:ManUpNewton}}
    \begin{algorithmic}[1]
        \Require Manifold $\M$, point $x\in\M$
        \State Find $d$ in $\tangentM$ that solves
        \begin{equation*}\tag{Newton equation}\label{eq:Newtoneq}
            \grad F(x) + \Hess F(x)[d] = 0
        \end{equation*}
        \State Find $\alpha$ satisfying the Armijo condition \eqref{eq:Armijo} with direction $d$
        \State \Return $y = R_x(\alpha d)$
    \end{algorithmic}
\end{algorithm}

\begin{theorem}\label{th:Newton}
    Let \cref{asm:fun} hold and take $\gamma\in(0, 1/L)$. Consider the sequence of iterates $(x_k)$ generated by \cref{alg:generic} equipped with the Riemannian Newton manifold update (\cref{alg:ManUpNewton}). If $~ \Hess \funtot(x_k)$ is positive definite at each step, then all limit points of $(x_k)$ are critical points of $\funtot$ and share the same functional value.

    Furthermore, assume that the sequence $(y_k)$ admits %
    a limit point $x^\star$ such\;that
    \begin{itemize}
        \item[i)] $x^\star\in\M$ is a $r$-structured critical point for $(f,g)$  with $r<1/\gamma$;
        \item[ii)] $\Hess_{\M} \funtot(x^\star)\succ 0$ and  $\Hess_{\M} \funtot$ is locally Lipschitz around $x^\star$.
    \end{itemize}
    Then, after some finite time,
    \begin{itemize}
        \item[a)] the sequence $(x_k)$ lies on $\M$;
        \item[b)] $x_k$ converges to $x^\star$ quadratically{: for large $k$, there exists $c>0$ such that}
            \begin{equation*}
                \distM(x_{k+1}, x^\star) \le {c ~} \distM(x_k, x^\star)^2.
            \end{equation*}
    \end{itemize}
\end{theorem}

\begin{proof}
    As the Riemannian Hessian is assumed to be positive definite, Newton's direction is a descent direction:
    \begin{equation*}
        \langle \grad \funtot(x_k), d_k \rangle = -\langle \grad \funtot(x_k), \Hess \funtot(x_k)^{-1} \grad \funtot(x_k) \rangle < 0.
    \end{equation*}
    The Riemannian Newton manifold step is therefore well-defined, and the \revise{line search} terminates by \cref{lemma:LS_exists_valid_stepsizes}, so that the manifold update is well-defined and provides descent \revise{($F(y_{k})\le F(x_{k})$)}.

    \revise{Now, since the proximal gradient update provides a descent (see \cite[Lem. 10.4]{beck2017first}),
    \begin{equation}\label{eq:proofglobalalg_descent}
        F(x_{k+1}) \le F(y_k) - \frac{1-\gamma L}{2\gamma} \| x_{k+1}-y_k \|^2 \le F(x_k) - \frac{1-\gamma L}{2\gamma} \| x_{k+1}-y_k \|^2.
    \end{equation}
    The sequence $(F(x_k))$ is thus non-increasing and lower-bounded, therefore it converges.
    Besides, any accumulation point of $(x_k)$ is a critical point of $F$. Indeed, summing equation \eqref{eq:proofglobalalg_descent} for $k=1, \ldots, n$ yields:
    \begin{equation*}
        \frac{1-\gamma L}{2\gamma} \sum_{k=1}^n \| x_{k+1}-y_k \|^2 \le F(x_1) - F(x_{n+1}) \le F(x_1)-\inf F < +\infty.
    \end{equation*}
    Since $ \dist (0, \partial F(x_{k+1})) \le \frac{L\gamma + 1}{\gamma}\|x_{k+1}-y_k\|$ (see e.g. the proof of\;\cite[Prop.\;13]{bolte2015error}), we have that 
    $\dist (0, \partial F(x_{k+1}))$
    converges to $0$. The outer-semi continuity of the limiting subdifferential then yields criticality of accumulation points.}

    Now we apply the local convergence of Riemannian Newton \cite[Th. 6.3.2]{absil2009optimization}:
    assumption ii) ensures that
    the Riemannian Newton direction $d$ computed in step\;1 of\;\cref{alg:ManUpNewton} provides a quadratic improvement %
    on a neighborhood of $x^\star$ on\;$\M$. Moreover, the \revise{line search} returns the unit-stepsize after some finite time: $\alpha=1$ is tried first, and is acceptable for directions providing superlinear improvement by \cref{lemma:LS_unit_stepsize}. Thus the whole Riemannian Newton update provides quadratic improvement after some finite time. Using this and assumption i), \cref{th:identiflocalcv_sperlin} applies and yields the results.
\end{proof}

This theorem states that alternating proximal gradient steps and Riemannian Newton steps %
converges quadratically %
to structured points with virtually the same assumptions the Euclidean Newton method. 
However, the two standard issues of Newton's method 
still hold in our setting: at each iteration, a linear system has to be solved to produce the Newton direction; and this direction does not always provide descent (without positive definiteness of the Hessian).
We show in the next section that %
truncated versions %
overcome these issues also in our framework.

\vspace*{-1ex}
\subsection{Riemannian Truncated Newton \& superlinear convergence}
\label{sec:truncnewt}

We consider a manifold update based on a truncated Newton procedure \cite{dembo1983truncated}. %
(Riemannian) Truncated Newton consists in solving \eqref{eq:Newtoneq} partially by using a (Riemannian) conjugate gradient procedure so that whenever the resolution of \eqref{eq:Newtoneq} is stopped, the resulting direction provides descent on the function. %
The quality of the truncated Newton direction is controlled by a parameter $\eta \in [0, 1)$ which bounds the ratio of residual and gradient norms:
\begin{equation}\tag{Inexact Newton eq.}\label{eq:inexactNewtonEquation}
    \| \grad F(x) + \Hess F(x)[d] \| \le \eta \| \grad F(x) \|.
\end{equation}

\vspace*{-1ex}
\begin{algorithm}
    \caption{$\manupdate$-Newton-CG\label{alg:ManUpTruncatedNewton}}
    \begin{algorithmic}[1]
        \Require Manifold $\M$, point $x\in\M$, convergence defining parameter $\theta\in(0,1]$
        \State Let $\eta = \|\grad F(x)\|^{\theta}$
        \State Find $d$ that solves \eqref{eq:inexactNewtonEquation}
        \State Find $\alpha$ satisfying the Armijo condition \eqref{eq:Armijo} with direction $d$
        \State \Return $y = R_x(\alpha d)$
    \end{algorithmic}
\end{algorithm}
\vspace*{-1ex}

\begin{theorem}
    Let \cref{asm:fun} hold and take $\gamma\in(0,1/L)$. Consider the sequence of iterates $(x_k)$ generated by \cref{alg:generic} equipped with the Riemannian Truncated Newton manifold update (\cref{alg:ManUpTruncatedNewton}). Then all limit points of $(x_k)$ are critical points of $\funtot$ and share the same function value.

    Furthermore, assume that sequence $(y_k)$ admits %
    a limit point $x^\star$ such\;that
    \begin{itemize}
        \item[i)] $x^\star\in\M$ is a $r$-structured critical point for $(f,g)$ with $r<1/\gamma$;
        \item[ii)] $\Hess_{\M} \funtot(x^\star)\succ 0$ and  $\Hess_{\M} \funtot$ is locally Lipschitz around $x^\star$.
        \item[iii)] we take $\eta_k = \mathcal O(\|\grad F(x_k)\|^\theta)$, for some $\theta \in (0, 1]$.
    \end{itemize}
    Then, for $k$ large enough, the full sequence $(x_k)$ lies on $\M$, and $x_k$ converges to $x^\star$ superlinearly with order\;$1+\theta$: for large\;$k$, there exist\;$c>0$,
    \begin{equation*}
        \distM(x_{k+1}, x^\star) \le c ~ \distM(x_{k}, x^\star)^{1+\theta}.
    \end{equation*}
\end{theorem}

\begin{proof}
    The direction provided by\;\eqref{eq:inexactNewtonEquation} is a descent direction by \cref{lemma:RiemannianNewtonCG_descent}, the \revise{line search} terminates by \cref{lemma:LS_exists_valid_stepsizes}, so that the %
    updates are well-defined and provide descent. Thus, \revise{as in the proof of \cref{th:Newton} we get that every accumulation point of the iterate sequence is a critical point for $\funtot$.} We can apply now the local convergence of the Riemannian truncated Newton method\;\cite[Th. 8.2.1]{absil2009optimization}: assumptions ii) and iii) ensure that 
    the direction $d$ computed in step\;1 of\;\cref{alg:ManUpTruncatedNewton} provides a local superlinear improvement towards $x^\star$. %
    The end of the proof is the same as the one of the proof of\;\cref{th:Newton}.
\end{proof}

\section{Numerical illustrations}
\label{sec:num}

In this section, we illustrate the effect of Newton acceleration. 
We consider \cref{alg:generic} equipped with either the Newton update of \cref{alg:ManUpNewton}, denoted `Alt. Newton' or the truncated Newton update of \cref{alg:ManUpTruncatedNewton}, denoted `Alt. Truncated Newton'. 
These methods are compared to the Proximal Gradient and the Accelerated Proximal Gradient, which serve as baseline.
The algorithms and problems are implemented in Julia \cite{bezanson2017julia}; experiments may be reproduced using the code available at \url{https://github.com/GillesBareilles/NewtonRiemannAccel-ProxGrad}.

We report the numerical results in figures showing a) the suboptimality $F(x_k)-F(x^\star)$ of the current iterate $x_k$ versus time, and b) the dimension of the current manifold $\M_k\ni x_k$ versus iteration.
We also report a table comparing the algorithms at the first iteration that makes suboptimality lower than tolerances $10^{-3}$ and $10^{-9}$ for various measures summarized in the following table: %

\medskip

    {\noindent
    \footnotesize
    \rowcolors{1}{gray!10}{white}
    \begin{tabular}{p{0.18\textwidth} p{0.75\textwidth}}
       $F(x_k)-F(x^\star)$ & Suboptimality at current iteration. \\
       \#prox.\;grad.\;steps & Number of proximal gradient steps, each involve computing $\nabla f(\cdot)$ and $\prox_{\gamma g}(\cdot)$ once. \\
       \#$\manupdate$ steps & Number of Riemannian steps, each involve computing $\grad F(\cdot)$ once and $\Hess F(\cdot)[\cdot]$ multiple times (one per Conjugate Gradient iteration). \\
       \#$\Hess F(\cdot)[\cdot]$ & Number of Riemannian Hessian-vector products, approximates the effort spent in manifold updates since algorithm started.\\
       \#$f$ & Number of calls to $f(x)$, one per iteration + some for the \revise{line search} + some for the backtracking estimation of the Lipschitz constant.\\
       \#$g$ & Number of calls to $g(x)$, one per iteration + some for the \revise{line search}. \\
    \end{tabular}
    }

\medskip

The proximal gradient updates, present in all methods, include a backtracking procedure that maintains an estimate of the Lipschitz constant of $\nabla f$, so that the proximal gradient \revise{step length} is taken as the inverse of that estimate. The Conjugate Gradient %
 used to solve \eqref{eq:Newtoneq} and \eqref{eq:inexactNewtonEquation} follows \cite[Alg. 6.2]{boumal2022intromanifolds}; it is stopped when the (in)exactness criterion is met, or after 50 iterations for the logistic problem and 150 for the trace-norm one, or when the inner direction $d$ makes the ratio $\langle\Hess F(x_k)[d], d\rangle / \|d\|^2$ small.\footnote{Each CG iteration requires one  Hessian-vector product, avoiding to form the Hessian matrix. \revise{A test on this ratio is used to detect a direction of quasi-negative curvature for the (Riemannian) Hessian, which is a stopping criterion of the Conjugate Gradient. In our implementation, we require this quantity to be smaller than\;$10^{-15}$\;for the Newton method. For the truncated version, we reduce the threshold when getting close to the solution: initialized at $1$, the threshold is decreased %
by a factor $10$ each time the unit-step is accepted by the line search.}} The manifold updates are completed by a backtracking \revise{line search} started from unit stepsize, a direct implementation of \cite[Alg.~6.3.1]{dennis1996numerical}.

\vspace*{-1ex}
\subsection{Two-dimensional nonsmooth example}
\label{sec:exp_2d}

We consider the piecewise quadratic problem of \cite{lewis2019simple}:
\begin{equation*}%
    \min_{x\in\bbR^2} 2x_1^2+x_2^2 + |x_1^2-x_2|.
\end{equation*}
The objective function is partly-smooth relative to the parabola %
$\{x \,:\, x_2=x_1^2\}$, for which an expression for the tangent space, the orthogonal projection on tangent space, a second-order retraction and conversion from Euclidean gradients and Hessian-vector products to Riemannian ones are readily available.

We run the proximal gradient, its accelerated counterpart, and \cref{alg:generic} with the Newton update \cref{alg:ManUpNewton}. The proximal gradient steps of all algorithms have a constant step-size $\gamma=0.05$, all algorithms are started from point $(2, 3)$.%

\begin{figure}[!h]
    \centering
    \begin{subfigure}[t]{0.6\textwidth}
        \centering
        \includegraphics[width=\textwidth]{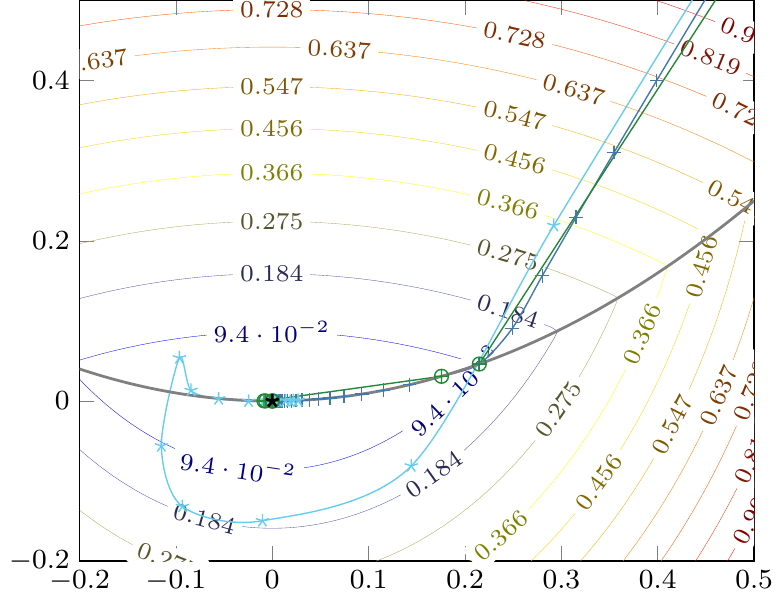}
    \end{subfigure}%
    \begin{subfigure}[t]{0.4\textwidth}
      \centering
      \raisebox{2.5cm}{
        \includegraphics[width=\textwidth]{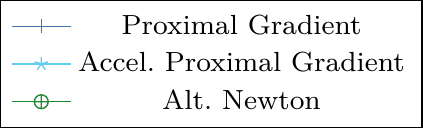}
      }
        \vspace*{-1cm}
    \end{subfigure} \\ \vspace{0.2cm}
    \begin{subfigure}[t]{\textwidth}
        \centering
        \resizebox{\textwidth}{!}{
          \begin{filecontents*}{maxquadAL.csv}
algo&tol&F-F*&proxgrad&gradF&HessF&f&g
Prox. Gradient&0.001&0.0007736370745816156&29&&&30&30
Prox. Gradient&1.0e-9&7.594823106800896e-10&60&&&61&61
Accel. Prox. Gradient&0.001&0.0009626158394158507&16&&&17&17
Accel. Prox. Gradient&1.0e-9&5.184441449855132e-10&63&&&64&64
Alt. Newton&0.001&0.00014946622308458022&2&2&10&7&7
Alt. Newton&1.0e-9&8.748444107374269e-13&3&3&15&10&10
\end{filecontents*}
\pgfplotstabletypeset{maxquadAL.csv}
        }
    \end{subfigure}
    \caption{nonsmooth example\label{fig:maxquad}}
\end{figure}

\medskip
\noindent\textbf{Observations} The iterates are displayed in \cref{fig:maxquad}. 
The Proximal Gradient iterates reach the parabola in finite time, and then converge linearly on the parabola while the Accelerated Proximal Gradient iterates ``overshoot'' the optimal manifold (see \cite{bareilles2019interplay}). 
The iterates of the Alt. Newton method stay on the parabola and the quadratic convergence behavior appears clearly since two Newton updates bring suboptimality below $10^{-3}$, and one additional step gets it below $10^{-12}$.

\vspace*{-1ex}
\subsection{$\ell_1$-regularized logistic problem} We now turn to the $\ell_1$-regularized logistic problem:
\begin{equation}\label{pb:logisticl1}
    \min_{x\in\bbR^n} \frac{1}{m} \sum_{i=1}^m \log(1+\exp(-y_i  \langle A_i, x \rangle))  + \lambda \|x\|_1,
\end{equation}
where $A\in\bbR^{m\times n}$, $y\in\{-1, 1\}^m$, and $\lambda>0$. The nonsmooth part $g(x)=\lambda\|x\|_1$ is described in Section\;\ref{sec:ex}.

We consider an instance where $n=4000$, $m=400$, $\lambda = 10^{-2}$ and the final manifold has dimension $249$. The coefficients of $A$ are drawn independently following a normal law. 
From a sparse random vector $s$, $y_i$ is set to $1$ with probability $1/(1+\exp(-\langle A_i , s\rangle))$, and $-1$ otherwise.  \revise{All algorithms start from the same point which is the output of $35$ iterations of the accelerated proximal gradient randomly initiated.}

\medskip
\noindent\textbf{Observations}
The experiments are presented in \cref{fig:logit}.
The optimal manifold is identified %
around iteration 200 for all methods except for Proximal Gradient, which needs 1000 iterations.
The two baselines Proximal Gradient and its accelerated version show linear convergence, with a better rate for the non accelerated version once the final manifold is reached. Alt. Truncated Newton shows superlinear acceleration, while Alt. Newton fails to converge in the given time budget.

\begin{figure}
    \centering
    \begin{subfigure}[t]{0.49\textwidth}
        \centering
        \includegraphics[width=\textwidth]{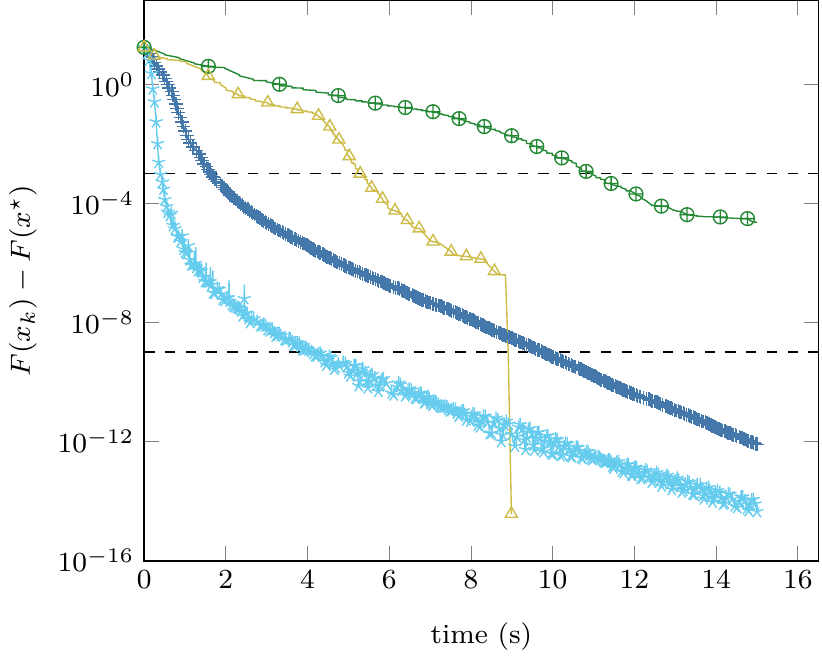}
        \caption{Suboptimality vs time (s)} \label{fig:numexps_logit_subopt}
    \end{subfigure}\hfill
    \begin{subfigure}[t]{0.49\textwidth}
        \centering
        \includegraphics[width=\textwidth]{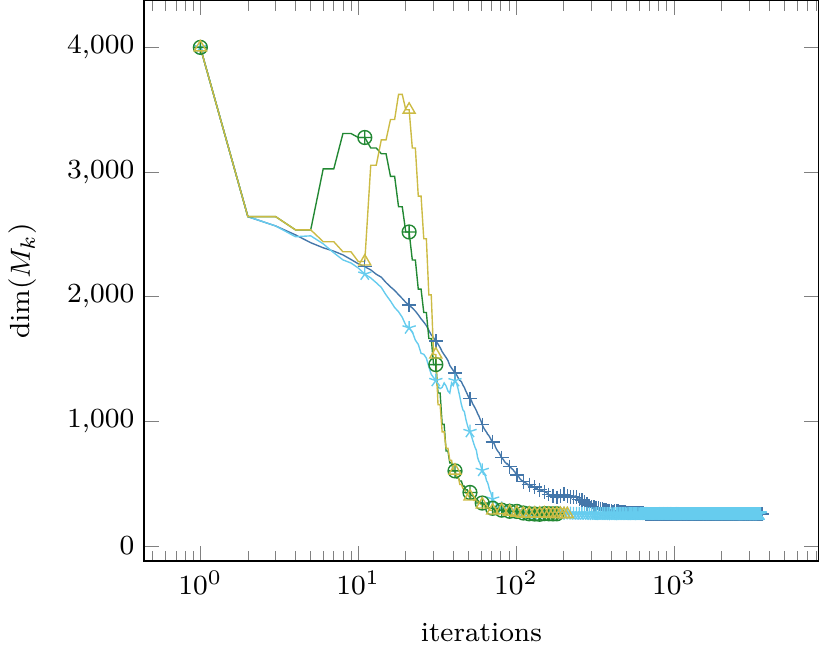}
        \caption{Manifold dimension vs iteration} \label{fig:numexps_logit_struct}
    \end{subfigure} \\ \vspace{0.2cm}
    \begin{subfigure}[t]{\textwidth}
        \includegraphics[width=\textwidth]{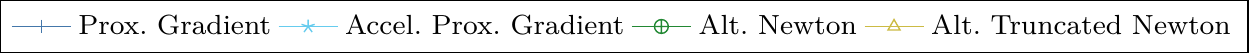}
    \end{subfigure} \\ \vspace{0.2cm}
    \begin{subfigure}[t]{\textwidth}
        \centering
        \resizebox{\textwidth}{!}{\begin{filecontents*}{logit-4000x400-elasticnet.csv}
algo&tol&F-F*&proxgrad&gradF&HessF&f&g
Prox. Gradient&0.001&0.0009963198229036019&357&&&779&358
Prox. Gradient&1.0e-9&9.965078207052613e-10&2306&&&4677&2307
Accel. Prox. Gradient&0.001&0.0009257766624239938&90&&&246&91
Accel. Prox. Gradient&1.0e-9&9.899422392933843e-10&953&&&1972&954
Alt. Newton&0.001&0.0009759231753842523&62&61&6303&556&427
Alt. Newton&1.0e-9&&&&&&
Alt. Truncated Newton&0.001&0.0009557819627238895&51&50&2616&437&321
Alt. Truncated Newton&1.0e-9&3.774758283725532e-15&105&105&5091&742&572
\end{filecontents*}
\pgfplotstabletypeset{logit-4000x400-elasticnet.csv}}
    \end{subfigure}
    \caption{Logistic-$\ell_1$ problem\label{fig:logit}}
\end{figure}

As iterations grow, the (Accelerated) Proximal Gradient identifies manifolds of decreasing dimension in a roughly monotonical way. Alt. Truncated Newton behaves differently: after identifying monotonically manifolds of dimension lower than $2000$, the dimension of the current manifold jumps to about $3000$ for about $10$ iterations, to finally reach quickly the final manifold. We believe that this partial loss of identified structure is caused by iterates getting close to a point
having one non-null but very small coordinate. There, the second-order Taylor extension is valid on a small set %
however it may lead to a Newton step that lies outside that set, thus driving the iterate away. %
The same behavior occurs for Alt. Newton. This difficulty can be related to the well-known problem of constraint activation in nonlinear programming. Despite this, %
\cref{alg:generic} retains a good rate overall.

\vspace*{-1ex}
\subsection{Trace-norm regularized problem}\label{sec:trace}

We consider the following matrix regression problem:
\begin{equation}\label{pb:tracenorm}
    \min_{x\in\bbR^{n_1\times n_2}} \frac{1}{2} \sum_{i=1}^m \left( \langle A_i, x\rangle - y_i \right)^2 + \lambda \|x\|_*,
\end{equation}
where $A_i\in\bbR^{n_1\times n_2}$ for $i=1, \ldots, m$, $y\in\bbR^m$ and $\lambda$ denotes a positive scalar. The nonsmooth part $g(x)=\lambda\|x\|_*$ is described in Section\;\ref{sec:ex}.

We consider an instance of \eqref{pb:tracenorm} where $n_1=10$, $n_2=12$, $m=60$, $\lambda = 10^{-2}$ and the final manifold is that of matrices of rank $6$. The coefficients of %
the $A_i$'s %
are drawn independently from %
a normal law.
From a sparse random vector $s$, 
$y_i$ is taken as $\langle A_i, s\rangle + \xi_i$, where $\xi_i$ follows a centered normal %
law with variance $0.01^2$. \revise{All algorithms start from the same point which is the output of $10^3$ iterations of the accelerated proximal gradient randomly initiated.}

\medskip
\noindent\textbf{Observations}
The experiments are presented in \cref{fig:tracenorm}. We see on \cref{fig:numexptracenorm_subopt} that the Proximal Gradient algorithm and its accelerated version converge sublinearly, which is to be related to the lack of strong convexity of the objective problem. Alt. Truncated Newton converges superlinearly, and shows the interest of the Newtonian acceleration. \Cref{fig:numexps_tracenorm_struct} shows that the Proximal Gradient does not reach the final optimal manifold within the budget of iterations; similarly for the Newton method, within the budget of time.

\begin{figure}
    \centering
    \begin{subfigure}[t]{0.49\textwidth}
        \centering
        \includegraphics[width=\textwidth]{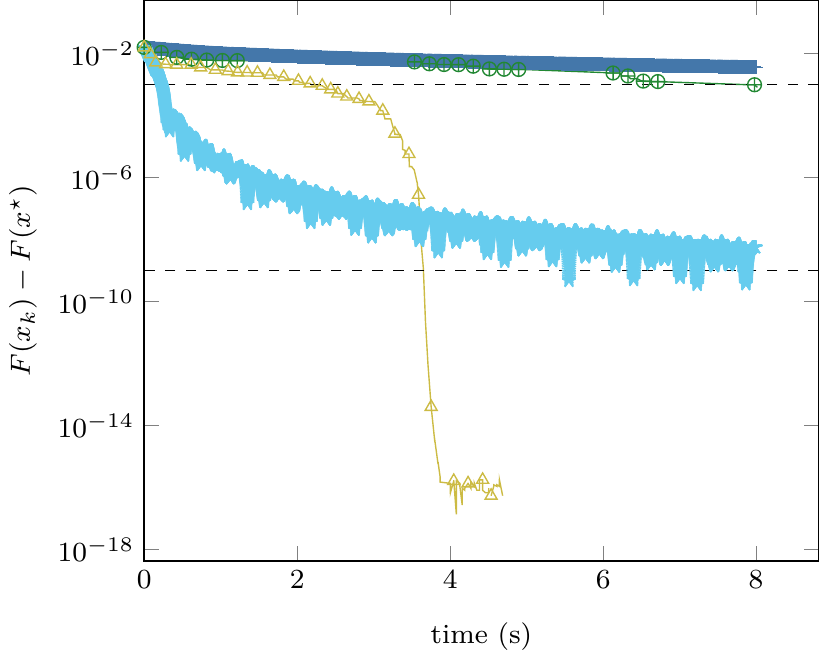}
        \caption{\vspace*{-2ex}Suboptimality vs time (s)} \label{fig:numexptracenorm_subopt}
    \end{subfigure}\hfill
    \begin{subfigure}[t]{0.49\textwidth}
        \centering
        \includegraphics[width=\textwidth]{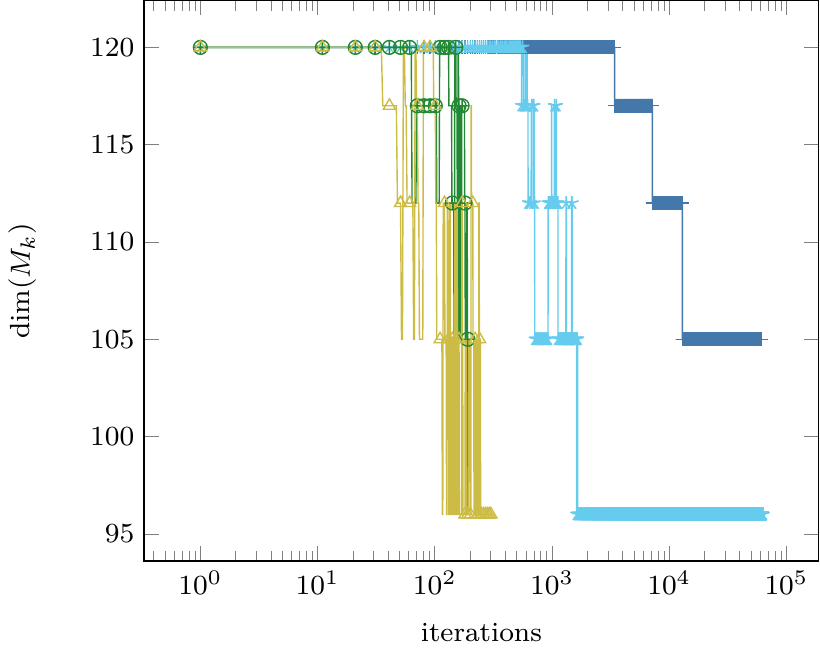}
        \caption{Manifold dimension vs iteration} \label{fig:numexps_tracenorm_struct}
    \end{subfigure} \\ \vspace{0.2cm}
    \begin{subfigure}[t]{\textwidth}
        \includegraphics[width=\textwidth]{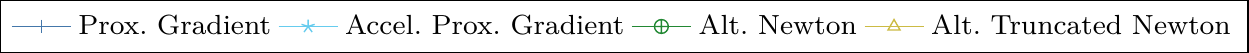}
    \end{subfigure} \\ \vspace{0.2cm}
    \begin{subfigure}[t]{\textwidth}
        \centering
        \resizebox{\textwidth}{!}{\begin{filecontents*}{tracenorm-10x12.csv}
algo&tol&F-F*&proxgrad&gradF&HessF&f&g
Prox. Gradient&0.001&&&&&&
Prox. Gradient&1.0e-9&&&&&&
Accel. Prox. Gradient&0.001&0.00099894916795637&1489&&&3073&1490
Accel. Prox. Gradient&1.0e-9&9.858174276899945e-10&43283&&&86661&43284
Alt. Newton&0.001&0.0009833250032105778&93&93&28063&873&687
Alt. Newton&1.0e-9&&&&&&
Alt. Truncated Newton&0.001&0.0009695009931029591&76&76&16342&738&568
Alt. Truncated Newton&1.0e-9&2.2716245551279712e-11&128&128&27786&1101&879
\end{filecontents*}
\pgfplotstabletypeset{tracenorm-10x12.csv}}
    \end{subfigure}
    \caption{Trace-norm problem\label{fig:tracenorm}\vspace*{-2ex}}
\end{figure}

\vspace*{-1ex}
\section{Concluding remarks}\label{sec:concludingrqs}

This paper proposes and studies a nonsmooth optimization algorithm
exploiting the underlying smooth geometry revealed by the proximal operator.
The %
method alternates between a proximal gradient step providing identification and a Riemann Newton acceleration providing superlinear convergence.
This algorithm has two special features: (\emph{i}) it does not rely on prior knowledge of the final manifold, and (\emph{ii}) its convergence is guaranteed in the (structured) nonconvex case.

Several extensions of this algorithm are possible;  %
specifically, both building blocks can be refined: other Newton accelerations could be considered (e.g. trust-region \cite{absil2007trust}, %
cubic regularization \cite{agarwal2018adaptive})
as well as other proximal algorithms (e.g. prox-Newton %
\cite{lee2014proximal}, fast proximal gradient \cite{beck2009fast}).
We focused here on the simplest Newton acceleration to highlight the ideas and the working horses of our approach.

\vspace*{-1ex}
\ifdraft
    \bibliographystyle{ormsv080}
\else
    \bibliographystyle{spmpsci}
\fi
\bibliography{references}

\appendix

\section{Preliminary Results on the Proximal Gradient}\label{sec:prox}

The first result shows the local Lipschitz continuity of the proximity operator.
It can be proven by applying \cite[Th. 4.4]{poliquin1996prox} with the assumption that $\bar{x}=\prox_{\bar\gamma \funns}(\bar{y})$, following the arguments of \cite[Th.~1]{hare2009computing}.
We provide here a self-contained proof.

\begin{lemma}
    \label{th:proxex}
    Consider a function $\funns:\bbR^n\to\bbR$, a pair of points $\bar{x}, \bar{y}$ and a \revise{step length} $\bar{\gamma}>0$ such that $\bar{x}=\prox_{\bar\gamma \funns}(\bar{y})$ and $\funns$ is $r$ prox-regular at $\bar{x}$ for subgradient $\bar{v} \triangleq (\bar{y}-\bar{x}) / \bar{\gamma}$.

    Then, for any $\gamma\in(0, \min(1/r, \bar\gamma))$, there exists a neighborhood $\N_{\bar{y}}$ of $\bar{y}$ over which $\prox_{\gamma \funns}$ is single-valued and $(1-\gamma r)^{-1}$-Lipschitz continuous. Furthermore, there holds
    \begin{equation*}
        x = \prox_{\gamma \funns}(y) \Leftrightarrow (y-x)/\gamma \in\partial \funns(x)
    \end{equation*}
    for $y\in\N_{\bar{y}}$ and $x$ near $\bar{x}$ in the sense $\|x-\bar{x}\|<\varepsilon$, $|\funns(x)-\funns(\bar{x})|<\varepsilon$ and $\|(y-x)/\gamma - \bar{v}\|<\varepsilon$. %
\end{lemma}

\begin{proof}
    One can easily check that prox-regularity of $\funns$ at $\bar{x}$ for subgradient $\bar{v}$ is equivalent to prox-regularity of function $\tilde{\funns}$ around $0$ for subgradient $0$, with $\tilde{\funns} = \funns(\cdot + \bar{x}) - \langle \bar{v}, \cdot\rangle - \funns(\bar{x})$ and a change of variable $\tilde{x} = x-\bar{x}$. Similarly, $\bar{x}=\prox_{\bar\gamma \funns}(\bar{y})$ is characterized by its global optimality~condition
    \begin{equation*}
        \funns(x) + \frac{1}{2\bar\gamma}\|x-\bar{y}\| > \funns(\bar{x}) + \frac{1}{2\bar\gamma}\|\bar{x}-\bar{y}\|^2 \qquad \text{ for all } x\neq\bar{x},
    \end{equation*}
    which we may write as
    \begin{equation*}
        \funns(x) > \funns(\bar{x}) + \langle \bar{v}, x-\bar{x} \rangle - \frac{1}{2\bar\gamma}\|x - \bar{x}\|^2 \qquad \text{ for all } x\neq\bar{x}.
    \end{equation*}
    Under that same change of variables, since $\tilde{\funns}(0) = 0$, this optimality condition rewrites as
    \begin{equation*}
        \tilde{\funns}(\tilde{x}) > - \frac{1}{2\bar\gamma}\|\tilde{x}\|^2 \qquad \text{ for all } \tilde{x}\neq 0.
    \end{equation*}
    We may thus apply Theorem 4.4 from \cite{poliquin1996prox} to get the claimed result on $\tilde{\funns}$, which transfers back to $\funns$ as our change of function and variable is bijective. We thus obtain that for $\gamma\in(0, \min(1/r, \bar\gamma))$, on a neighborhood $\N_{\bar{y}}$ of $\bar{y}$, $\prox_{\gamma \funns}$ is single-valued, $(1-\gamma r)^{-1}$-Lipschitz continuous and $\prox_{\gamma \funns}(y)=[I+\gamma T]^{-1}(y)$, where $T$ denotes the $\funns$-attentive $\varepsilon$-localization of $\partial \funns$ at $\bar{x}$. Taking $y$ near $\bar{y}$ and $x$ near $\bar{x}$ such that $\|x-\bar{x}\|<\varepsilon$, $|\funns(x)-\funns(\bar{x})|<\varepsilon$ and $\|(y-x)/\gamma - \bar{v}\|<\varepsilon$ allows to identify the localization of $\partial \funns(x)$ with $\partial \funns(x)$, so that
    \begin{equation*}
        \frac{y-x}{\gamma}\in \partial \funns(x)\Leftrightarrow \frac{y-x}{\gamma}\in T(x)\Leftrightarrow (I+\gamma T)(x)=y\Leftrightarrow x=\prox_{\gamma \funns}(y).
    \end{equation*}
    Note that the proof of \cite[Th. 4.4]{poliquin1996prox} includes a minor error relative to the Lipschitz constant computation, we report here a corrected value.
\end{proof}

Now, we show that critical points of prox-regular functions are strong local minimizers; this result appears more or less explicitly in some articles, including\;\cite{daniilidis2006geometrical}.

\begin{lemma}\label{lemma:prox_strong_min}
    Let $f$ and $g$ denote two functions and $\bar{x}$, $\bar{y}$ two points such that $f$ is differentiable at $\bar{y}$ and $g$ is $r$-prox-regular at $\bar{x}$ for subgradient $\frac{1}{\gamma}(\bar{y}-\bar{x}) - \nabla f(\bar{y}) \in \partial g(\bar{x})$ with $\gamma \in (0, 1/r)$.
    Then, the function $\rho_{\bar{y}} : x\mapsto g(x) + \frac{1}{2\gamma}\|\bar{y}-\gamma\nabla f(\bar{y}) - x\|^2$ satisfies %
    \begin{equation*}
        \rho_{\bar{y}}(x) \ge \rho_{\bar{y}}(\bar{x}) + \frac{1}{2} \left( \frac{1}{\gamma} - r \right) \| x-\bar{x} \|^2, \qquad \text{for all $x$ near $\bar{x}$.}
    \end{equation*}
\end{lemma}

\begin{proof}
    Prox-regularity of $g$ at $\bar{x}$ with subgradient $\frac{1}{\gamma} (\bar{y}-\gamma \nabla f(\bar{y}) - \bar{x}) \in \partial g(\bar{x})$ writes
    \begin{equation*}
        g(x) \ge g(\bar{x}) + \frac{1}{\gamma}\langle \bar{y}-\gamma \nabla f(\bar{y}) - \bar{x} , x-\bar{x} \rangle - \frac{r}{2} \| x-\bar{x} \|^2.
    \end{equation*}
    The identity $ 2\langle b-a, c-a\rangle = \|b-a\|^2 + \|c-a\|^2 -   \|b-c\|^2 $ applied to the previous scalar product yields:
    \begin{equation*}
        g(x) \ge g(\bar{x}) + \frac{1}{2\gamma} \| \bar{y}-\gamma \nabla f(\bar{y}) - \bar{x} \|^2 + \frac{1}{2\gamma} \| x-\bar{x} \|^2 - \frac{1}{2\gamma} \|\bar{y}-\gamma \nabla f(\bar{y}) - x \|^2 - \frac{r}{2} \| x-\bar{x} \|^2,
    \end{equation*}
    which rewrites
    \begin{equation*}
        \underbrace{ g(x) + \frac{1}{2\gamma} \|\bar{y}-\gamma \nabla f(\bar{y}) - x \|^2 }_{=\rho_{\bar{y}}(x) } \ge  \underbrace{  g(\bar{x}) + \frac{1}{2\gamma} \| \bar{y}-\gamma \nabla f(\bar{y}) - \bar{x} \|^2   }_{=\rho_{\bar{y}}(\bar{x}) } + \frac{1}{2} \left( \frac{1}{\gamma} - r \right) \| x-\bar{x} \|^2,
    \end{equation*}
    \vspace*{-1ex}
which is the claimed inequality.
\end{proof}

\section{Technical results on Riemannian methods.}\label{sec:riemann}

In this section, we provide basic results on Riemannian optimization that simplify our developments and that we have not been able to find in the existing literature.

\vspace*{-1ex}
\subsection{Euclidean spaces and manifolds, back and forth}

We establish here a connection between the Riemannian and the Euclidean distances.

\begin{lemma}\label{lem:InvRetractionApproxLog}
    Consider a point $\bar{x}$ of a Riemannian manifold $\M$, equipped with a retraction $\R$ such that $\R_{\bar{x}}$ is $\C^2$. For any $\varepsilon>0$, there exists a neighborhood $\U$ of $\bar{x}$ in $\M$ such that 
    \begin{equation*}
        (1-\varepsilon)\distM(x, \bar{x}) \le \|\R^{-1}_{\bar{x}}(x)\| \le (1+\varepsilon)\distM(x, \bar{x}) \qquad \text{ for all $x\in\U$.}
    \end{equation*}
    where $\R^{-1}_{\bar{x}}:\M\to\tangent{\bar{x}}{\M}$ is the smooth inverse of $\R_{\bar{x}}$ defined locally around $\bar{x}$.
\end{lemma}
\begin{proof}
    The retraction at $\bar{x}$ can be inverted locally around $0$. Indeed, as $\D\R_{\bar{x}}(0_{\tangent{\bar{x}}{\M}})=I$ is invertible and $\R_{\bar{x}}$ is $\C^2$, the implicit function theorem provides the existence of a $\C^2$ inverse function $\R^{-1}_{\bar{x}}:\M\to\tangent{\bar{x}}{\M}$ defined locally around $\bar{x}$. Furthermore, one shows by differentiating the relation $\R_{\bar{x}}\circ\R_{\bar{x}}^{-1}$ that the differential of $\R^{-1}_{\bar{x}}$ at $\bar{x}$ is the identity.

    We consider the function $f:\M\to\RR$ defined by $f(x)=\|\log_{\bar{x}}(x)\|-\|\R_{\bar{x}}^{-1}(x)\|$. Clearly $f(\bar{x})=0$, and $\D f(\bar{x})=0$ as the differentials of both $\R_{\bar{x}}^{-1}$ and logarithm at $\bar{x}$ are the identity. In local coordinates $\hat{x}=\log_{\bar{x}}{x}$ around $\bar{x}$, $f$ is represented by the function $\hat{f}=f\circ\exp_{\bar{x}}: \tangent{\bar{x}}{\M}\to\RR$. As $\hat{f}(\hat{\bar{x}})=0$, $\D \hat{f}(\hat{\bar{x}})=0$ and $\hat{f}$ is $\C^2$, there exists some $C>0$ such that
    \begin{equation*}
        -C \|\hat{x} - \hat{\bar{x}}\|^2 \le \hat{f}(\hat{x}) \le C \|\hat{x} - \hat{\bar{x}}\|^2 \qquad\text{in a neighborhood $\hat{\U}$ of $\hat{\bar{x}}$,}.
    \end{equation*}
    For any $\varepsilon>0$, by taking a small enough neighborhood $\hat{\U}'\subset\hat{\U}$, there holds
    \begin{equation*}
        -\varepsilon \|\hat{x} - \hat{\bar{x}}\| \le \hat{f}(\hat{x}) \le \varepsilon \|\hat{x} - \hat{\bar{x}}\|.
    \end{equation*}
    Thus for all $x$ in $\U = \R_{\bar{x}}(\hat{\U}')$,
    \begin{equation*}
        -\varepsilon \|\log_{\bar{x}}(x)\| \le \|\log_{\bar{x}}(x)\|-\|\R_{\bar{x}}^{-1}(x)\| \le \varepsilon \|\log_{\bar{x}}(x)\|,
    \end{equation*}
    as $\hat{x}=\log_{\bar{x}}(x)$, $\hat{\bar{x}}=0$. We conclude with %
    $\distM(x, \bar{x})=\|\hat{x} - \hat{\bar{x}}\| = \|\log_{\bar{x}}(x)\|$.
\end{proof}

We recall a slightly specialized version of \cite[Th. 2.2]{miller2005newton}, which is essentially the application of the implicit function theorem around a point of a manifold.

\begin{proposition}\label{prop:tangretraction}
    Consider a $p$-dimensional $\C^k$-submanifold $\M$ of $\bbR^n$ around a point $\bar{x}\in\M$. The mapping $\R : \tangentBundle\to\M$, defined for $(x, \eta)\in\tangentBundle$ near $(\bar{x}, 0)$ by $\proj_x(\R(x, \eta)) = \eta$ defines a second-order retraction near $(\bar{x}, 0)$. The point-wise retraction, defined as $\R_x = \R(x, \cdot)$, is locally invertible with inverse $\R_{x}^{-1}=\proj_{x}$.
\end{proposition}

\begin{proof}
    Let $\Psi : \bbR^n\to\bbR^{n-p}$ denote a $\C^k$ function defining $\M$ around $\bar{x}$: for all $x$ close enough to $\bar{x}$, there holds $x\in\M \Leftrightarrow \Psi(x)=0$, and $\D\Psi(x)$ is surjective. Consider the equation $\Phi(x, \eta_t, \eta_n) = 0$ around $(\bar{x}, 0, 0)$, with
    \begin{equation*}
        \begin{array}{rl}
            \Phi : \left\{ x, \eta_t, \eta_n : x\in\M, \eta_t\in\tangentM, \eta_n\in\normalM \right\} & \to \bbR                       \\
            x, \eta_t, \eta_n                                                                         & \mapsto \Psi(x+\eta_t+\eta_n).
        \end{array}
    \end{equation*}
    The partial differential $\D_{\eta_n}\Phi(\bar{x}, 0, 0)$ is, for $\xi_n\in\normalM[\bar{x}]$,
    \begin{equation*}
        \D_{\eta_n}\Phi(\bar{x}, 0, 0)[\xi_n] = \D\Psi(\bar{x})[\xi_n].
    \end{equation*}
    Since $\bar{x}\in\M$, $\D_{\eta_n}\Phi(\bar{x}, 0, 0)$ is surjective from $\normalM[\bar{x}]$ to $\bbR^{n-p}$ so its a bijection. The implicit function theorem provides the existence of neighborhoods $\N_{\bar{x}}^1\subset\M$, $\N_{0}^2\subset\cup_{x\in\M} \tangentM$, $\N_{0}^3\subset\cup_{x\in\M}\normalM$ and a unique $\C^k$ function $\eta_n : \N_{\bar{x}}^1 \times \N_0^2 \to \N_0^3$ such that, for all $x\in\N_{\bar{x}}^1$, $\eta_t\in\N_0^2$ and $\eta_n\in\N_0^3$, $\eta_n(\bar{x}, 0) = 0$ and
    \begin{equation*}
        \Phi(x, \eta_t, \eta_n(x, \eta_t)) = 0 \Leftrightarrow x+\eta_t+\eta_n(x, \eta_t) \in \M.
    \end{equation*}
    It also provides an expression for the partial derivative of $\eta_n$ at $(x, 0)$ along $\eta_t$: for $\xi_t\in\tangentM$,
    \begin{equation*}
        \D_{\eta_t} \eta_n(x, 0)[\xi_t] = - \left[ \D_{\eta_n}\Phi(x, 0, 0) \right]^{-1} \D_{\eta_t} \Phi(x, 0, 0)[\xi_t].
    \end{equation*}
    As noted before, $\D_{\eta_n}\Phi(x, 0, 0)$ is bijective since $x\in\M$. Besides, $\D_{\eta_t} \Phi(x, 0, 0) = \D\Phi(x)[\xi_t]=0$ since $\tangentM$ identifies as the kernel of $\D\Phi(x)$. Thus $\D_{\eta_t} \eta_n(x, 0) = 0$.

    Now, define a map $\R:\N_{\bar{x}}^1\times\N_0^2 \to \M$ by $\R(x, \eta_t) = x + \eta_t + \eta_n(x, \eta_t)$. This map has degree of smoothness $\C^k$ since $\eta_n$ is $\C^k$, satisfies $\R(x, 0) = x$ since $\eta_n(x, 0)=0$ and satisfies $\D_{\eta_t}\eta_n(x, 0) = I + \D_{\eta_t}(x, 0) = I$. Thus $\R$ defines a retraction on a neighborhood of $(\bar{x}, 0)$.

    We turn to show the second-order property of $\R$. Consider the smooth curve $\smoothcurve$ defined as $c(t) = \R(x, t\eta)$ for some $x\in\N_{\bar{x}}^1$, $\eta_t\in\tangentM\cap\N_0^2$. It's first derivative writes
    \begin{equation*}
        c'(t) = \eta + \D_{\eta_t}\eta_n(x, t\eta)[\eta] = \eta.
    \end{equation*}
    The acceleration of the curve $c$ is obtained by computing the derivative of $c'(\cdot)$ in the ambient space and then projecting onto $\tangentM$. Thus $c''(t)=0$ and in particular, $c''(0)=0$ which makes $\R$ a second-order retraction.
\end{proof}

\begin{lemma}\label{lemma:Euclidean_to_Riemannian}
    Consider a point $\bar{x}$ of a Riemannian manifold $\M$. For any $\varepsilon>0$, there exists a neighborhood $\U$ of $\bar{x}$ in $\M$ such that, for all $x\in\U$,
    \begin{equation*}
        (1-\varepsilon)\distM(x, \bar{x}) \le \|x-\bar{x}\| \le (1+\varepsilon)\distM(x, \bar{x}),
    \end{equation*}
    where $\|x-\bar{x}\|$ is the Euclidean distance in the ambient space.
\end{lemma}
\begin{proof}
    Let $\bar{x}$, $x$ denote two close points on $\M$. Consider the tangential retraction introduced in \cref{prop:tangretraction}. As a retraction, it satisfies:
    \begin{equation*}
        \R_{\bar{x}}(\eta) = \R_{\bar{x}}(0) + \D\R_{\bar{x}}(0)[\eta] + \mathcal O(\|\eta\|^2) = \bar{x} + \mathcal O(\|\eta\|^2).
    \end{equation*}
    Taking $x=\R_{\bar{x}}(\eta)$ allows to write $x = \bar{x} + \mathcal O(\|\R_{\bar{x}}^{-1}(x)\|^2)$, so that for any small $\varepsilon_1>0$, there exists a small enough neighborhood $\U_1\subset\U$ of $\bar{x}$ in $\M$ such that
    \begin{equation*}
        (1-\varepsilon_1) \|\R_{\bar{x}}^{-1}(x)\| \le \|x-\bar{x}\| \le (1+\varepsilon_1)\|\R_{\bar{x}}^{-1}(x)\|.
    \end{equation*}
    By \cref{lem:InvRetractionApproxLog}, for $\varepsilon_2>0$ small enough, there exists a neighborhood $\U_2\subset\U$ of $\bar{x}$ such that,
    \begin{equation*}
        (1-\varepsilon_2)\distM(x, \bar{x}) \le \|\R^{-1}_{\bar{x}}(x)\| \le (1+\varepsilon_2)\distM(x, \bar{x}).
    \end{equation*}
    With $\varepsilon_1$, $\varepsilon_2$ such that $1-\varepsilon=(1-\varepsilon_1)(1-\varepsilon_2)$, we combine the two estimates to conclude.%
\end{proof}

\subsection{Two technical results on Riemannian descent algorithms}\label{appendix:soundnessRiemanLS}

We provide here two technical results used in the proofs of Section\;\ref{sec:specificationsManUp}.
First, \Cref{th:unitstepsuperlinearcv} adapts \cite[Th. 4.16]{bonnans2006numerical} to the Riemannian setting.
Second, \Cref{{lemma:RiemannianNewtonCG_descent}} 
adapts the proof of \cite[Lemma A.2]{dembo1983truncated} to the Riemannian setting.

\begin{theorem}[Soundness of the Riemannian \revise{line search}]\label{th:unitstepsuperlinearcv}
    Consider a manifold $\M$ equipped with a retraction $\R$ and a twice differentiable function $\funman:\M\to\RR$ that admits a strong local minimizer $x^\star$, that is, a point such that $\Hess \funman(x^\star)$ is positive definite. If $x$ is close to $x^\star$, $\eta$ brings a superlinear improvement towards $x^\star$, that is $\distM(\R_x(\eta), x^\star) = o(\distM(x, x^\star))$ as $x\to x^\star$, and $0<m_1<1/2$, then $\eta$ is acceptable by the Armijo rule\;\eqref{eq:Armijo} with unit stepsize $\alpha=1$.
\end{theorem}

\begin{proof}

    Let $x, \eta\in\tangentBundle$ denote a pair such that $x$ is close to $x^\star$ and $\distM(\R_x(\eta), x^\star) = o(\distM(x, x^\star))$. For convenience, let $x_+ = \R_x(\eta)$ denote the next point.

    Following \cite{absil2009optimization} (see e.g. the proof of Th. 6.3.2), we work in local coordinates around $x^\star$, representing any point $x\in\M$ by $\widehat{x} = \log_{x^\star}(x)$ and any tangent vector $\eta\in\tangentM$ by $\widehat{\eta}_x = \D \log_{x^\star}(x)[\eta]$. The function $\funman$ is represented by $\widehat{\funman} = \funman\circ\exp_{x^\star}:\tangentM[x^\star]\to\bbR$. Defining the coordinates via the logarithm grants the useful property that the Riemannian distance of any two points $x, y\in\M$ matches the euclidean distance between their representatives: $\distM(x, y) = \|\widehat{x}-\widehat{y}\|$. Besides, there holds
    \begin{equation}\label{eq:gradHess_to_localcoords}
        \D \funman(x)[\eta]=\D \widehat{\funman}(\widehat{x})[\widehat{\eta}] \quad \text{ and } \quad \Hess \funman(x)[\eta, \eta]=\D \widehat{\funman}(\widehat{x})[\widehat{\eta}, \widehat{\eta}].
    \end{equation}
    Indeed, $\D \funman(x)[\eta] = (\funman\circ\geocurve)'(0)$ and $\Hess \funman(x)[\eta, \eta] = (\funman\circ\geocurve)''(0)$, where $\geocurve$ denotes the geodesic curve defined by $\widehat{\geocurve}(t) = \widehat{x} + t\widehat{\eta}$. Using $\funman\circ\geocurve=\widehat{\funman}\circ\widehat{\geocurve}$, one obtains the result.

    \underline{Step 1.} We derive an approximation of $\D \funman(x)[\eta] = \langle \grad \funman(x), \eta\rangle$ in terms of $\D^2 \widehat{\funman}(\widehat{x^\star})[\widehat{x}-\widehat{x^\star}]^2$. To do so, we go through the intermediate quantity $\D \widehat{\funman}(\widehat{x}) [\widehat{x_+}-\widehat{x}]$, and handle precisely the $o(\cdot)$ terms. By smoothness of $\widehat{\funman}$ and since $\D \widehat{\funman}(\widehat{x^\star})=0$, Taylor's formula for $\D \widehat{\funman}$ writes
    \begin{align*}
        \D \widehat{\funman}(\widehat{x}) [\widehat{x_+}-\widehat{x}] &= \D^2 \widehat{\funman}(\widehat{x^\star})[\widehat{x_+}-\widehat{x}, \widehat{x}-\widehat{x^\star}] + o(\|\widehat{x}-\widehat{x^\star}\|^2)\\
        &= -\D^2 \widehat{\funman}(\widehat{x^\star})[\widehat{x}-\widehat{x^\star}]^2 + \D^2 \widehat{\funman}(\widehat{x^\star})[\widehat{x_+}-\widehat{x^\star}, \widehat{x}-\widehat{x^\star}] + o(\|\widehat{x}-\widehat{x^\star}\|^2)\\
        &= -\D^2 \widehat{\funman}(\widehat{x^\star})[\widehat{x}-\widehat{x^\star}]^2 + o(\|\widehat{x}-\widehat{x^\star}\|^2),
    \end{align*}
    where, in the last step, we used that $\|\widehat{x_+}-\widehat{x^\star}\| = o(\|\widehat{x}-\widehat{x^\star}\|)$ to get that $\|\D^2 \widehat{\funman}(\widehat{x^\star})[\widehat{x_+}-\widehat{x^\star}, \widehat{x}-\widehat{x^\star}]\| = \|\D^2 \widehat{\funman}(\widehat{x^\star})\| \|\widehat{x_+}-\widehat{x^\star}\|\|\widehat{x}-\widehat{x^\star}\| = o(\|\widehat{x}-\widehat{x^\star}\|^2)$.
    We now turn to show that $\D \widehat{\funman}(\widehat{x}) [\widehat{x_+}-\widehat{x}]$ behaves as $\D \funman(x)[\eta]$ up to $o(\|\widehat{x_+}-\widehat{x}\|^2)$. Since $\D \funman(x)[\eta]=\D\widehat{\funman}(\widehat{x})[\widehat{\eta}]$ by \eqref{eq:gradHess_to_localcoords}, there holds:
    \begin{equation*}
        \|\D \funman(x)[\eta] - \D\widehat{\funman}(\widehat{x})[\widehat{x_+} - \widehat{x}]\| = \|\D\widehat{\funman}(\widehat{x})[\widehat{\eta} - (\widehat{x_+} - \widehat{x})]\| \le \|\D\widehat{\funman}(\widehat{x}) \| \|\widehat{\eta} - (\widehat{x_+} - \widehat{x})\|.
    \end{equation*}

    As $\funman$ is twice differentiable and $\exp$ is $\C^\infty$, $\widehat{\funman}$ is twice differentiable as well. In particular its derivative is locally Lipschitz continuous, so that for $\widehat{x}$ near $\widehat{x^\star}$, we obtain a first estimate:
    \begin{equation*}
        \|\D\widehat{\funman}(\widehat{x})\| = \|\D\widehat{\funman}(\widehat{x}) - \D\widehat{\funman}(\widehat{x^\star})\| = \mathcal O (\| \widehat{x}-\widehat{x^\star} \|).
    \end{equation*}

    Besides, the following estimate holds $\|\widehat{\eta} - (\widehat{x_+} - \widehat{x})\| = o (\| \widehat{x}-\widehat{x^\star} \|)$. Indeed, as the function $\log_{x^\star}\circ\R_x : \tangentM\to\tangentM[x^\star]$ is differentiable, there holds for $\eta\in\tangentM$ small,
    \begin{equation*}
        \log_{x^\star}(\R_x(\eta)) = \log_{x^\star}(\R_x(0)) + \D\log_{x^\star}(\R_x(0))[\D\R_x(0)[\eta]] + o(\|\eta\|),
    \end{equation*}
    which simplifies to $\widehat{x_+} = \widehat{x} + \widehat{\eta} + o(\|\eta\|)$. \Cref{lem:InvRetractionApproxLog} allows to write $\|\eta\| = \|\R_x^{-1}(x_+)\| = \mathcal O(\distM(x, x_+))$. Using the triangular inequality and the assumption that $\distM(x_+, x^\star) = o(\distM(x, x^*))$ we get
    \begin{equation*}
        \distM(x, x_+) \le \distM(x, x^\star) + \distM(x^\star, x_+) = \mathcal O(\distM(x, x^*)) = \mathcal O(\|\widehat{x}-\widehat{x^\star}\|),
    \end{equation*}
    so that the second estimate holds.

    Combining the two above estimates allows to conclude that
    \begin{equation*}
        \|\D \funman(x)[\eta] - \D\widehat{\funman}(\widehat{x})[\widehat{x_+} - \widehat{x}]\| = o(\|\widehat{x}-\widehat{x^\star}\|^2),
    \end{equation*}
    so that overall,
    \begin{equation*}
        \D \funman(x)[\eta] = \D \widehat{\funman}(\widehat{x}) [\widehat{x_+}-\widehat{x}] + o(\|\widehat{x}-\widehat{x^\star}\|^2) = -\D^2 \widehat{\funman}(\widehat{x^\star})[\widehat{x}-\widehat{x^\star}]^2 + o(\|\widehat{x}-\widehat{x^\star}\|^2).
    \end{equation*}

    Using that $\|\widehat{x}-\widehat{x^\star}\| = \distM(x, x^\star)$ and $\D^2 \widehat{\funman}(\widehat{x^\star}) = \Hess \funman(x^\star)$ \eqref{eq:gradHess_to_localcoords}, we obtain
    \begin{equation}\label{eq:proofunitstepsize_gradhess}
        \D \funman(x)[\eta] = -\Hess \funman(x^\star)[\widehat{x}-\widehat{x^\star}]^2 + o(\distM(x, x^\star)^2).
    \end{equation}

    \underline{Step 2.} The function $\funman$ admits a second-order development around $x^\star$: applying \cref{eq:manifoldSecondOrderDevSecRetraction} with the exponential map $\exp_{x^\star}$ as a second-order retraction yields
    \begin{equation}\label{eq:proofunitstepsize_fdev}
        \funman(x) = \funman(x^\star) + \frac{1}{2} \Hess \funman(x^\star)[\widehat{x}-\widehat{x^\star}]^2 + o(\distM(x, x^\star)^2),
    \end{equation}
    where we used that $\distM(x, x^\star) = \|\log_{x^\star}(x)-\log_{x^\star}(x^\star)\|$. Denote $0 < l \le L$ the lower and upper eigenvalues of $\Hess \funman(x^\star)$. The combination $\eqref{eq:proofunitstepsize_fdev} + m_1 \eqref{eq:proofunitstepsize_gradhess}$ writes
    \begin{align*}
        \funman(x) + m_1 \D \funman(x)[\eta] &= \funman(x^\star) + (\frac{1}{2}-m_1) \Hess \funman(x^\star)[\widehat{x}-\widehat{x^\star}]^2 + o(\distM(x, x^\star)^2) \\
        &\ge \funman(x^\star) + (\frac{1}{2}-m_1) l \distM(x, x^\star)^2 + o(\distM(x, x^\star)^2),
    \end{align*}

    Let $\varepsilon >0$ such that $\frac{1}{2}L \varepsilon^2  < (\frac{1}{2}-m_1) l$. As $\distM(x_+, x^\star)=o(\distM(x, x^\star))$, for $x$ close enough to $x^\star$ there holds $\distM(x_+, x^\star) \le \varepsilon \distM(x, x^\star)$. Combining this with the second-order development of $f$ at $x_+$, there holds:
    \begin{align*}
        \funman(x_+) &= \funman(x^\star) + \frac{1}{2} \Hess \funman(x^\star)[\widehat{x_+}-\widehat{x^\star}]^2 + o(\distM(x_+, x^\star)^2) \\
        &\le \funman(x^\star) + \frac{1}{2}L \distM(x_+, x^\star)^2 + o(\distM(x_+, x^\star)^2)\\
        &\le \funman(x^\star) + \frac{1}{2}L \varepsilon^2 \distM(x, x^\star)^2 + o(\distM(x, x^\star)^2).
    \end{align*}

    Subtracting the two estimates yields
    \begin{equation*}
        \funman(x_+) - (\funman(x) + m_1 \D \funman(x)[\eta]) \le \left(\frac{1}{2}L \varepsilon^2 - (\frac{1}{2}-m_1) l \right) \distM(x, x^\star)^2 + o(\distM(x, x^\star)^2),
    \end{equation*}
    which ensures that the Armijo condition is satisfied.

\end{proof}

\begin{lemma}[Riemannian Newton-CG a descent direction]\label{lemma:RiemannianNewtonCG_descent}
    Let \cref{asm:fun} hold and consider a manifold $\M$ and a point $x\in\M$. If $\funtot$ is twice differentiable on $\M$ at $x$ and $x$ is not a stationary point of $\funtot$, then there holds:
    \begin{equation*}%
        \langle \grad \funtot(x), d \rangle \le -\min(1, \|\Hess \funtot(x)\|^{-1}) \|\grad \funtot(x)\|^2,
    \end{equation*}
    where $d$ was obtained solving \eqref{eq:inexactNewtonEquation} with any forcing parameter $\eta$.
\end{lemma}
\begin{proof}
    The result is obtained by applying the analysis of \cite[Lemma A.2]{dembo1983truncated} to the approximate resolution of \eqref{eq:inexactNewtonEquation} on the euclidean space $\tangentM$, with constant specified according to the proof.
\end{proof}

\section{Complements to the experimental section}

\subsection{Oracles of \cref{sec:exp_2d}}

We detail here the oracles of $\funspart(x)\triangleq 2x_1^2+x_2^2$ and $\funnspart(x) \triangleq |x_1^2-x_2|$:
\begin{itemize}
    \item \emph{proximity operator:} For $\gamma<1/2$, there holds
    \begin{align*}
        \prox_{\gamma \funnspart}(x) = \begin{cases}
            (\frac{x_1}{1+2\gamma}, x_2+\gamma) &\text{ if } x_2 \le \frac{x_1^2}{(1+2\gamma)^2} - \gamma \\
            (\frac{x_1}{1+4\gamma t-2\gamma}, x_2 + 2\gamma t - \gamma) &\text{ if } \frac{x_1^2}{(1+2\gamma)^2} - \gamma \le x_2 \le \frac{x_1^2}{(1-2\gamma)^2} + \gamma \\
            (\frac{x_1}{1-2\gamma}, x_2-\gamma) &\text{ if } \frac{x_1^2}{(1-2\gamma)^2} + \gamma \le x_2
        \end{cases}
    \end{align*}
    where $t$ solves $x_2^2 + (-2\gamma t + \gamma - x_2)(1+4\gamma t -2\gamma)^2 = 0$.
    \item \emph{Riemannian gradient and Hessian:} Since $\funnspart$ is identically null on $\M$, for any point $(x, \eta)\in\tangentBundle$, $
        \grad \funnspart(x) = 0  \text{ and }  \Hess \funnspart(x)[\eta] = 0
    $. Moreover, Euclidean gradient and Hessian-vector product are converted to Riemannian ones using equations \eqref{eq:egrad_to_rgrad} and \eqref{eq:ehess_to_rhess}:
    \begin{align*}
        \grad \funspart(x) &= \proj_{x}(\nabla \funspart(x)) \\
        \Hess \funspart(x)[\eta] &= \proj_{x}\left( \nabla^2 \funspart(x)[\eta] - \begin{pmatrix}
            2\eta_1 \\ 0
        \end{pmatrix} \left\langle \nabla \funspart(x), \begin{pmatrix}
            2x_1 \\ -1
        \end{pmatrix}\right\rangle \frac{1}{1+4x_1^2} \right),
    \end{align*}
    and the orthogonal projection onto $\tangentM$ writes
    \begin{equation*}
        \proj_{x}(d) = d - \left\langle d, \begin{pmatrix}
            2x_1 \\ -1
        \end{pmatrix}\right\rangle \frac{1}{1+4x_1^2} \begin{pmatrix}
            2x_1 \\ -1
        \end{pmatrix}.
    \end{equation*}
\end{itemize}

\subsection{Differentiating the singular-value decomposition}\label{sec:diffsvd}

We establish the expressions of the derivative of the matrices involved in the singular value decomposition. These results may be seen as part of folklore, but, up to our knowledge, there are not explicitly written in the literature. We need them for the computations related to trace-norm regularized problems.

\begin{lemma}\label{lem:diffsvd}
    Consider the manifold of fixed rank matrices $\M_r$, a pair $x, \eta\in\tangentBundle$ and a smooth curve $\smoothcurve : I\to\M_r$ such that $\smoothcurve(0)=x$, $\smoothcurve'(0)=\eta$. Besides, let $U(t)$, $\Sigma(t)$, $V(t)$ denote smooth curves of $St(m, r)$, $\RR^{r\times r}$, $St(n, r)$ such that $\gamma(t) = U(t)\Sigma(t) V(t)^\top$. The derivatives of the decomposition factors at $t=0$ write
    \begin{align*}
        U'      & = U \left( F \circ \left[ U^\top \eta V \Sigma + \Sigma V^\top \eta^\top U  \right]\right) + (I_m - U U^\top) \eta V \Sigma^{-1}     \\
        V'      & = V \left( F \circ \left[ \Sigma U^\top \eta V + V^\top \eta^\top U \Sigma  \right]\right) + (I_n - V V^\top) \eta^\top U\Sigma^{-1} \\
        \Sigma' & = I_k \circ \left[ U^\top \eta V \right],
    \end{align*}
    where $I_k$ is the identity of $\RR^{k\times k}$, $\circ$ denotes the Hadamard product and $F\in\RR^{r\times r}$ is such that $F_{ij} = 1 / (\Sigma_{jj}^2 - \Sigma_{ii}^2)$ if $\Sigma_{jj} \neq \Sigma_{ii}$, and $F_{ij}=0$ otherwise. Equivalently, when the tangent vector is represented as $\eta = UMV^\top + U_pV^\top + UV_p^\top$, the above expressions simplify to
    \begin{align*}
        U'      & = U \left( F \circ \left[ M \Sigma + \Sigma M^\top  \right]\right) + U_p \Sigma^{-1}     \\
        V'      & = V \left( F \circ \left[ \Sigma M + M^\top \Sigma  \right]\right) + V_p \Sigma^{-1} \\
        \Sigma' & = I_k \circ M,
    \end{align*}
\end{lemma}

\begin{proof}
    We consider the curve $\gamma$ and all components and derivatives at $t=0$, therefore we don't mention evaluation time. Differentiating $\gamma = U\Sigma V^\top$ yields
    \begin{equation}\label{eq:svddiffDerGlobalCurve}
        \eta = U'\Sigma V^\top + U\Sigma' V^\top + U\Sigma V'^\top
    \end{equation}
    As a tangent vector to the Stiefel manifold at point $U$, $U'$ can be expressed as \cite[Ex. 3.5.2]{absil2009optimization}
    \begin{equation}\label{eq:svddiffUderDecomp}
        U'=U\Omega_U + U_\perp B_U,
    \end{equation}
    where $\Omega_U\in\RR^{r\times r}$ is a skew-symmetric matrix, $B_U\in\RR^{m-r \times m-r}$, and $U_\perp$ is any matrix such that $U^\top U_\perp = 0$ and $U_\perp^\top U_\perp = I_{m-r}$.
    Similarly, $V' = V\Omega_V + V_\perp B_V$, where $\Omega_V\in\RR^{r\times r}$ is skew-symmetric, $B_V\in\RR^{n-r \times n-r}$, and $V_\perp$ is any matrix such that $V^\top V_\perp = 0$ and $V_\perp^\top V_\perp = I_{n-r}$.

    Computing $U^\top \times \eqref{eq:svddiffDerGlobalCurve} \times V$ yields
    \begin{equation*}
        U^\top \eta V = \Omega_U\Sigma + \Sigma' + \Sigma \Omega_V^\top.
    \end{equation*}
    Looking at the diagonal elements of this equation yields the derivative of the diagonal component of $\eta$. This is done by taking the Hadamard product of both sides of previous equation with the identity matrix of $\RR^{r\times r}$, and writes
    \begin{equation*}%
        \Sigma' = I_r\circ \left[ U^\top \eta V \right].
    \end{equation*}
    The off-diagonal elements of this equation write
    \begin{equation} \label{eq:svddiffOffDiag}
        \bar I_r\circ \left[ U^\top \eta V \right] = \Omega_U\Sigma + \Sigma \Omega_V^\top,
    \end{equation}
    where $\bar I_r$ %
    has zeros on the diagonal and ones elsewhere. Adding $\eqref{eq:svddiffOffDiag}\Sigma$ and $\Sigma\eqref{eq:svddiffOffDiag}^\top$ yields
    \begin{equation*}
        \bar I_r\circ \left[ U^\top \eta V \Sigma + \Sigma V^\top \eta^\top U \right] = \Omega_U\Sigma^2 - \Sigma^2\Omega_U,
    \end{equation*}
    which decouples coefficient-wise. At coefficient $(ij)$, with $i\neq j$,
    \begin{equation*}
        \left[ U^\top \eta V \Sigma + \Sigma V^\top \eta^\top U \right]_{ij} = [\Omega_U]_{ij}\left(\Sigma^2_{jj} - \Sigma^2_{ii}\right),
    \end{equation*}
    hence $\Omega_U = F \circ \left[ U^\top \eta V \Sigma + \Sigma V^\top \eta^\top U \right]$, where $F\in\RR^{m-r\times r}$ has zeros on the diagonal and for $i\neq j$, $F_{ij} = 1/(\Sigma^2_{jj} - \Sigma^2_{ii})$ if $\Sigma^2_{jj} \neq \Sigma^2_{ii}$, $0$ otherwise.
    Besides, left-multiplying \eqref{eq:svddiffDerGlobalCurve} by $U_\perp^\top$ yields $U_\perp^\top \eta = U^\top_\perp U'\Sigma V^\top$, which rewrites, using the decomposition \eqref{eq:svddiffUderDecomp} of $U'$, as $U_\perp^\top \eta = B_U \Sigma V^\top$. Hence $B_U = U_\perp^\top \eta V \Sigma^{-1}$ and we get the complete expression for $U'$ by assembling the expressions of $\Omega_U$ and $B_U$ with the decomposition \eqref{eq:svddiffUderDecomp}. The term $U_\perp^\top U_\perp$ is eliminated using that $U^\top U + U_\perp^\top U_\perp = I_m$.

    Let's follow the same steps to get expressions for $V'$. Adding $\Sigma\eqref{eq:svddiffOffDiag}$ and $\eqref{eq:svddiffOffDiag}^\top \Sigma$ yields
    \begin{equation*}
        \bar I_r\circ \left[ \Sigma U^\top \eta V + V^\top \eta^\top U \Sigma \right] = \Omega_V \Sigma^2 - \Sigma^2 \Omega_V,
    \end{equation*}
    from which we get $\Omega_V = F \circ \left[ \Sigma U^\top \eta V + V^\top \eta^\top U \Sigma\right]$. Besides, right-multiplying \eqref{eq:svddiffDerGlobalCurve} by $V_\perp$ yields $\eta V_\perp = U \Sigma V'^\top V_\perp$, which rewrites using the decomposition $V' = V\Omega_V + V_\perp B_V $ as $\eta V_\perp = U \Sigma B_V^\top$. Hence $B_V = V_\perp^\top \eta^\top U \Sigma^{-1}$, and we get the claimed formula by eliminating the $V_\perp$ terms with $V^\top V + V_\perp^\top V_\perp = I_n$.
    The simplified expressions are obtained using that $U^\top U=I_m$, $U^\top U_p=0$, $V^\top V=I_n$ and $V^\top V_p=0$.
\end{proof}

We are now ready to give the expression of the Riemannian gradient and Hessian of the nuclear norm.

\begin{proposition}\label{prop:nuclearnorm_rgradhess}
    The nuclear norm $g=\|\cdot\|_*$ restricted to $\M_r$ is $\C^2$ and admits a smooth second-order development of the form \eqref{eq:manifoldSecondOrderDevSecRetraction} near any point $x = U\Sigma V^\top\in\M_r$. Denoting $\eta=U M V^\top + U_p V^\top + U V_p^\top\in\tangentM_r$ a tangent vector, there holds:
    \begin{align*}
        \grad g(x) &= U V^\top\\
        \Hess g(x)[\eta] &= U \left[\tilde{F} \circ (M-M^\top) \right] V^\top + U_p \Sigma^{-1}V^\top + U\Sigma^{-1}V_p^T,
    \end{align*}
    where $\circ$ denotes the Hadamard product and $\tilde{F}\in\RR^{r\times r}$ is such that $\tilde{F}_{ij} = 1 / (\Sigma_{jj} + \Sigma_{ii})$ if $\Sigma_{jj} \neq \Sigma_{ii}$, and $\tilde{F}_{ij}=0$ otherwise.
\end{proposition}

\begin{proof}
    Let $\smoothcurve : I\to\M_r$ denote a smooth curve over $\M_r$ such that $\gamma(0)=x$ and $\gamma'(0)=\eta$, and consider $\varphi = \|\smoothcurve(\cdot)\|_* : I\to\RR$. Writing the decomposition $\smoothcurve(t) = U(t)\Sigma(t) V(t)^\top$, for $U(t)$, $\Sigma(t)$, $V(t)$ smooth curves of $St(m, r)$, $\RR^{r\times r}$, $St(n, r)$ allows to write $\varphi(t) = \Tr(\Sigma(t))$. Applying \Cref{lem:diffsvd} yields
    \begin{equation*}
        \varphi'(0) = \Tr(\Sigma'(0)) = \Tr(U^\top \eta V) = \Tr(\eta  V U^\top) = \langle \eta, UV^\top \rangle,
    \end{equation*}
    so that $\grad g(x) = UV^\top \in \tangentM[X]$.

    In order to obtain the Riemannian Hessian, let $\bar Z : I\to\RR^n$ denote a smooth extension of $\grad g(\smoothcurve(\cdot))$, defined by $\bar Z(t) = U(t)V(t)^\top$. The Riemannian Hessian is then obtained as $\Hess g(x)[\eta] = \proj_x \bar Z'(0)$. The derivative of $\bar Z$ at $0$ is simply $\bar Z'(0) = U' V^\top + U V'^\top$ and thus writes, applying  \cref{lem:diffsvd}
    \begin{equation*}
        \bar Z'(0) = U \left( F \circ \left[ M \Sigma + \Sigma M^\top  \right]\right) V^\top + U_p \Sigma^{-1} V^\top +
        U \left( F \circ \left[ \Sigma M + M^\top \Sigma  \right]\right)^\top V^\top + U \Sigma^{-1} V_p^\top
    \end{equation*}
    This expression simplifies to the statement by using the fact that $F$ is antisymmetric and applying the identity $(A\circ B)^\top = A^\top \circ B^\top$.
\end{proof}

\subsection{Additional numerical experiment}
\label{subsec:additionalnumexps}

\begin{wrapfigure}{r}{0.5\textwidth}
    \centering
    \includegraphics[width=.48\textwidth]{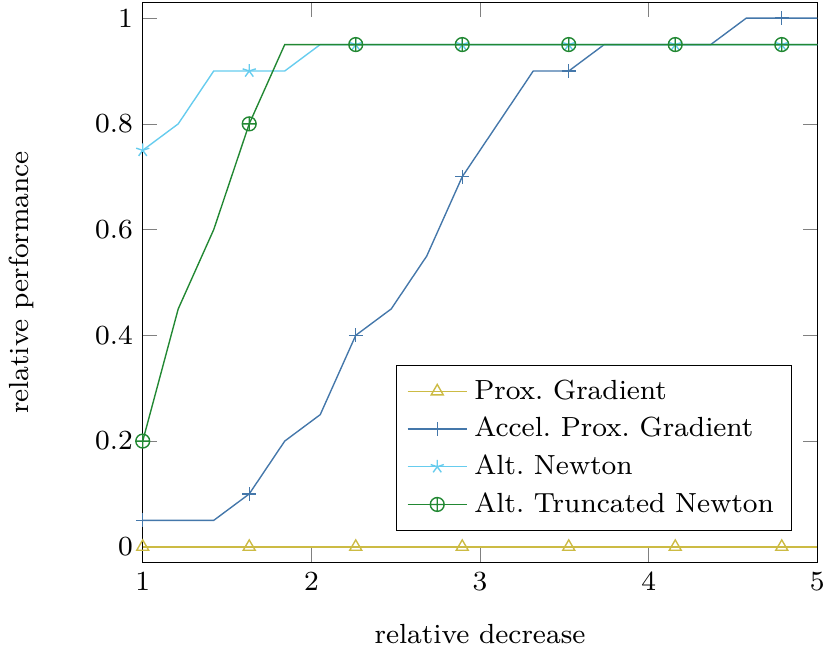}
    \caption{\revise{Performance profile for the time to decrease suboptimality below $10^{-9}$}\label{fig:tracenorm-perfprofile}}
\end{wrapfigure}

\revise{We illustrate in this appendix the robustness of the Newton acceleration on several instances of the same problem. More precisely, in the set-up of Section\;\ref{sec:trace}, we compare the 4 algorithms on 20 random instances of the tracenorm problem, in terms of wallclock time required to reach a suboptimality of $10^{-9}$. We then provide in \cref{fig:tracenorm-perfprofile} a performance profile (i.e. the ordinate of a curve at absciss $t \ge 1$ indicates the proportion of problems for which the corresponding algorithm was able to satisfy the criterion within $t$ times the best algorithm time for each problem; see\;\cite{dolan2002benchmarking}). 

We observe the following on\;\cref{fig:tracenorm-perfprofile}.
The ordinate at origin of a curve gives the proportion of problems for which the corresponding algorithm performed best: methods with Newton acceleration are the most efficient in $95\% (= 75\%+20\%)$ of the instances. Furthermore, they require about $2.5\times$\;less time to converge in half of the instances. Note also that the proximal gradient is completely outperformed by the others algorithms since it takes $5\times$\;more time than the best algorithm, for all instances. 

}

\end{document}